\documentclass[11pt,a4paper]{article}
\usepackage[utf8]{inputenc}
\usepackage[english]{babel}
\usepackage[T1]{fontenc}
\usepackage{amsmath}
\usepackage{amsthm}
\usepackage{amsfonts}
\usepackage{amssymb}
\usepackage{subcaption}
\usepackage{graphicx}

\usepackage{hyperref}
\usepackage{enumitem}

\usepackage[table]{xcolor}
\usepackage{tikz}

\definecolor{darkgreen}{rgb}{0,0.6,0}

\definecolor{darkblue}{rgb}{0,0,0.7} 
\definecolor{darkred}{rgb}{0.9,0.1,0.1}
\newcommand{\R}{\mathbb{R}}
\newcommand{\N}{\mathbb{N}}
\newcommand{\C}{\mathbb{C}}

\renewcommand{\div}{\mathrm{div}\, }

\newcommand{\ra}{\rightarrow}

\newcommand{\bc}{\begin{center}}
	\newcommand{\ec}{\end{center}}

\newtheorem{thm}{Theorem}[section]
\newtheorem{cor}[thm]{Corollary}
\newtheorem{Lemma}[thm]{Lemma}

\newtheorem{Proposition}[thm]{Proposition}
\theoremstyle{definition}

\theoremstyle{remark}
\newtheorem{remark}[thm]{Remark}

\providecommand{\keywords}[1]
{
	\small	
	\textbf{\textit{Keywords--}}
}

\title{On the instability of travelling wave solutions for the Transport-Stokes equation}

\author{M. Bonnivard\footnote{Ecole Centrale de Lyon, CNRS, INSA Lyon, Universite Claude Bernard Lyon 1, Université Jean Monnet, ICJ UMR5208,
		69130 Ecully, France,
		matthieu.bonnivard@ec-lyon.fr}, A. Mecherbet\footnote{Université Paris Cité and Sorbonne Université, CNRS, IMJ-PRG, F-75013 Paris, France. mecherbet@imj-prg.fr
}}

\begin{document}
	
	\maketitle	
	\begin{abstract}
		In this paper, we investigate the instability of the spherical travelling wave solutions for the Transport-Stokes system in $\R^3$. First, a  classical scaling argument ensures instability among all probability measures for the Wasserstein metric and the $L^1$ norm. Secondly, we address the instability among patch solutions with a perturbed surface. To this end, we study the linearized system of a contour dynamics equation derived in \cite{Mecherbet2020} in the case where the support of the patch is axisymmetric and described by spherical parametrization. We investigate numerically the existence of positive eigenvalues, which ensures the instability of the linearized system. Eventually we recover numerically the instability of the travelling wave by solving the Transport-Stokes equation using a finite element method on FreeFem.
	\end{abstract}
	\begin{keywords}{}   
		Stokes flow; transport equation; suspensions; instability of travelling wave solutions; numerical simulations
	\end{keywords}
	%	76D07 stokes & related
	%35Q49 transport
	%76T20 suspensions
	%	35P15  	Estimates of eigenvalues in context of PDEs
	% 35C07  	Travelling wave solutions
	%76M10  	Finite element methods applied to problems in fluid mechanics
	
	\section{Introduction}

	In this paper, we consider the \emph{Transport-Stokes equation} which describes sedimentation of small rigid inertial spherical particles in a viscous fluid and is given by
	\begin{eqnarray}
		- \Delta u +\nabla p =- \rho\,  e_3 \quad \textrm{in }\R ^3, \label{StokesR3} \\
		\div u = 0  \quad \textrm{in }\R ^3, \label{Incomp_R3}\\
		\partial_t \rho + \div(\rho u) =0 \quad \textrm{in }\R ^3, \label{Transport_R3}\\
		\rho(t=0)= \rho_0 \quad \textrm{in }\R ^3, \label{Init_rho_R3}\\
		\lim_{|x|\ra \infty} u(x) = 0, \label{limInfini}
	\end{eqnarray}
	where $\rho_0$ is a probability density in $\R^3$. This system has been derived in \cite{Hofer18,Mecherbet2019} as a mean field limit of a large number of particles sedimenting in a Stokes flow. Its well-posedness has been investigated in several papers. We refer to~\cite{Hofer18} for a first well-posedness result for regular initial data. In~\cite{Mecherbet2020}, the author shows global existence and uniqueness of Lagrangian solutions in the set $\mathcal{P}_1(\R^3)\cap L^\infty(\R^3)$ of bounded probability measures with finite first moment, and persistence of patch solutions. In~\cite{Hofer&Schubert}, authors prove global existence and uniqueness of Lagrangian solutions in $\mathcal{P}(\R^3)\cap L^\infty(\R^3)$. Global existence and uniqueness for $L^\infty$ initial data in the case of bounded domains in dimension $2$ and $3$, as well as in an infinite strip in dimension $2$ has been proven in~\cite{Leblond}. We also refer to~\cite{Mecherbet&Sueur} where well-posedness and controllability are established in $\mathcal{P}(\R^3)\cap L^p(\R^3)$, $p\geq3$ and analyticity of the trajectories in case $p>3$. In~\cite{GrayerII}, the author investigates the $2$-dimensional case for compactly supported initial data and proves persistence of patch solutions in this setting, as well as in the case of a torus. Finally, we refer to the recent results~\cite{Cobb} where well-posedness of the transport-fractional Stokes system is investigated, and to~\cite{Inversi} where the author proves the existence of Lagrangian solution for merely $L^1$ initial data, and shows uniqueness in a Yudovich-type space.

	It has been shown by Hadamard~\cite{Hadamard} and Rybczynski~\cite{Rybczynski} in 1911 that spherical patch solutions are travelling wave solutions for the Transport-Stokes system. Such a result has been investigated in the papers~\cite{MNG, EMG} where authors observe through physical experiments that a spherical viscous droplet falling inside a lighter viscous fluid slowly evolves to a torus before breaking up in smaller droplets which reproduce the same scenario. This phenomenon is also observed by solving numerically the system of ODEs describing sedimenting particles at a microscopic level. We emphasize that such a microscopic model has been shown to converge to the Transport-Stokes system as a mean field on any finite time intervals, see \cite{Hofer18,Mecherbet2019}. These results suggest that a slight perturbation of the spherical patch leads to a fundamental different behaviour when time evolves. 

 \medskip 
 In this paper we are interested then in proving the instability of the travelling wave. To our knowledge there is no such result concerning the Transport-Stokes system in this setting. Such stability issues have been investigated in \cite{Gancedo&Granero-Belinchon} for the flat stationary state in the $2$D horizontally periodic strip $\mathbb{T}\times \R$. We also refer to the recent paper \cite{DGL} where stability of decreasing stratified density profile is proven in a periodic channel $\mathbb{T}\times (0,1)$.
	
	We first emphasize that, thanks to a classical scaling argument, any spherical patch of initial density remains spherical with a constant velocity sedimentation that depends on the radius of the droplet. This shows the instability of spherical patch solutions among probability measures in terms of the Wasserstein metric as well as the $L^1$ norm, see Proposition \ref{prop:instability_TS}. 
	
	We consider next the case where we initially perturb  the support of the droplet. This is done by considering the surface evolution of the droplet $1_{B_0}$ through a hyperbolic model derived in~\cite{Mecherbet2020} in the case of axisymmetric initial domains $B_0$ that can be described through a spherical parametrization. { We develop and implement a numerical method to study the stability of the linearized model around the traveling wave solution, relying on the reformulation of the relevant eigenvalue problem as a fixed point problem (see Propositions~\ref{prop_T_lamda} and~\ref{prop_carac_fonction_propre_A}). The results that we obtain point out to the existence of a positive eigenvalue, hence to the instability of the spherical droplet in this context.
 }
	
	Finally, we recover numerically the instability of the travelling wave by solving the Transport-Stokes system using a finite element method on FreeFem. We first reproduce the droplet sedimentation in a vessel that is experimented in \cite{MNG}, see Figure~\ref{Fig:freefall}.  We next consider the truncated domain approach with artificial boundary conditions, see Figure~\ref{Fig:res_nonborne}. We observe the formation of a torus in both simulations, which is reminiscent of the results of \cite{MNG}.

	\section{Instability of patch solutions}\label{Sec:patch_instab}

	\begin{Proposition}
		If $\rho \in L^\infty\cap L^1$ satisfies $\partial_t \rho +\nabla \rho \cdot \mathcal{K}\star \rho=0 $ with $\mathcal{K}=-\mathcal{U}e_3$, then 
		$$
		\bar{\rho}(t,x):=\frac{1}{\lambda^{\alpha+\beta}}\rho\left(\frac{t}{\lambda^\beta}, \frac{x}{\lambda^\alpha}\right)
		$$
		is also a solution for any $\alpha, \beta ,\lambda \in \R\times\R\times\R^\star_+$.
	\end{Proposition}
	\begin{proof}
		This follows from the fact that $$\left(\mathcal{K}\star \bar{\rho}\right)(t,x)= \lambda^{\alpha-\beta}\mathcal{K}\star \rho \left( \frac{t}{\lambda^\beta},\frac{x}{\lambda^\alpha}\right).$$
		Indeed, using $\mathcal{K}(\lambda x)= \mathcal{K}(x)/\lambda$, we get
		\begin{align*}
			\left(\mathcal{K}\star \bar{\rho}\right)(t,x)&= \int_{\R^3} \mathcal{K}(x-y) \bar{\rho}(t,y) dy \\
			&=\frac{1}{\lambda^{\alpha+\beta}}\int_{\R^3} \mathcal{K}(x-y) \rho\left(\frac{t}{\lambda^\beta}, \frac{y}{\lambda^\alpha}\right) dy  \\
			&=\frac{\lambda^{3\alpha}}{\lambda^{\alpha+\beta}}\int_{\R^3} \mathcal{K}\left(\lambda^\alpha\left(\frac{x}{\lambda^\alpha}-z \right)\right) \rho\left(\frac{t}{\lambda^\beta}, z\right) dy \\
			&=\frac{\lambda^{2\alpha}}{\lambda^{\alpha+\beta}}\int_{\R^3} \mathcal{K}\left(\frac{x}{\lambda^\alpha}-z \right) \rho\left(\frac{t}{\lambda^\beta}, z\right) dy.
		\end{align*}
		This yields 
		\begin{align*}
			& \left[\partial_t \bar{\rho}+\nabla \bar{\rho} \cdot \mathcal{K}\star \bar{\rho}\right](t,x)\\
			&=\left[\frac{1}{\lambda^{{\alpha+2\beta}}}\partial_t \rho +\frac{1}{\lambda^{2\alpha+\beta}} \nabla \rho\cdot \lambda^{\alpha-\beta}\mathcal{K}\star \rho   \right]\left(\frac{t}{\lambda^\beta}, \frac{x}{\lambda^\alpha} \right)=0\, .
		\end{align*}
	\end{proof}
	\begin{cor}\label{Cor:explicit_formula_rho_R}
		Let $R>0$ and $\rho_{R,0}(x)=\frac{1_{B(0,R)}}{|B(0,R)|} $, then the unique solution of the Transport-Stokes system with initial data $\rho_{R,0}$ is
		\begin{equation}\label{Formula:rho_R(t,x)}
			\rho_{R}(t,x)=\frac{1}{R^3}\rho_{1,0}\left(\frac{x}{R}-\frac{t}{|B(0,1)| R^2}c\right)
		\end{equation}
		with $c=-\frac{4}{15}e_3$
	\end{cor}
	\begin{proof} We recall that the unique solution to the Transport-Stokes system for $\rho_0(x)=1_{B(0,1)}$ is given by $\rho: (t,x)\mapsto \rho_0(x-ct)$ with $c=-\frac{4}{15}e_3$. Hence, by the above proposition,
		$$
		\bar{\rho}(t,x):=\frac{1}{|B(0,1)|}\rho\left(\frac{t}{|B(0,1)|},{x}\right)
		$$
		satisfies also a Transport-Stokes equation for the initial condition $\bar{\rho}(0,x)=\frac{1}{|B(0,1)|}\rho\left(0,{x}\right)=\frac{1}{|B(0,1)|}\rho_0\left({x} \right)=\rho_{1,0}(x)$. Thus, by uniqueness,
		\begin{multline*}
			\rho_{1}(t,x)=\frac{1}{|B(0,1)|}\rho\left(\frac{t}{|B(0,1)|},{x}\right)= \frac{1}{|B(0,1)|}\rho_0\left(x-c\frac{t}{|B(0,1)|}\right)\\
			=\rho_{1,0}\left(x-c\frac{t}{|B(0,1)|}\right).
		\end{multline*}
		Analogously,
		$$
		\tilde{\rho}(t,x)= \frac{1}{R^3} \rho_1\left(\frac{t}{R^2},\frac{x}{R}\right)
		$$
		satisfies the Transport-Stokes equation with initial condition 
		$$
		\tilde\rho(0,x)= \frac{1}{R^3} \rho_1\left(0,\frac{x}{R}\right)=\frac{1}{R^3} \rho_{1,0}\left(\frac{x}{R}\right)=\rho_{R,0}(x),
		$$
		hence by uniqueness,
		$$
		\rho_{R}(t,x)=\frac{1}{R^3}\rho_1\left(\frac{t}{R^2},\frac{x}{R}\right)= \frac{1}{R^3} \rho_{1,0}\left(\frac{x}{R}-c \frac{t}{R^2 |B(0,1)|}\right).
		$$
	\end{proof}
	\begin{Proposition}\label{prop:instability_TS}
		Let $R>0$.	 Any travelling wave type solution $\rho_R $ is unstable for the first Wasserstein metric.
	\end{Proposition}
	\begin{proof}
		We restrict to the case $R=1$ for simplicity but the proof for any $R$ is analogous.
		The idea is to prove the following claim: 
		\begin{multline*}
			\exists \,\epsilon>0, \forall\, \eta>0, \exists \bar{\rho}_0 \in \mathcal{P}(\R^3)\cap L^\infty(\R^3)\,, \exists T>0 \,  \text{ such that}\\
			W_1(\bar{\rho}_0,\rho_{1,0}) <\eta \text{ and } W_1(\bar{\rho}(t),\rho_1(t)) >\epsilon,\, \forall\, t \geq T
		\end{multline*}
		where $\bar{\rho}$ the unique solution associated to $\bar{\rho}_0$. 
		
		Setting $\bar{\rho}_0=\rho_{R,0}$ and using the relation $\rho_{R,0}(x)=\frac{1}{R^3} \rho_{1,0}(x/R)$, we obtain
		\begin{align*}
			W_1(\rho_{R,0},\rho_{1,0}) &= \underset{\underset{[f]_{\text{Lip}=1}}{f \in \text{Lip}(\R^3)}}{\sup} \left|\int f(x) \rho_{1,0}(x)dx - \int f (x)\rho_{R,0}(x)dx \right|\\
			&=\underset{\underset{[f]_{\text{Lip}=1}}{f \in \text{Lip}(\R^3)}}{\sup} \left|\int f \rho_{1,0} -\int f (Rx)\rho_{1,0}\left({x}\right)dx \right|\\
			&= \underset{\underset{[f]_{\text{Lip}=1}}{f \in \text{Lip}(\R^3)}}{\sup}  \left|\int \left(f(x)-f\left({x}{R}\right) \right) \rho_{1,0}(x) dx\right| \\
			&\leq \left|1-{R} \right| \left(\int |x| \rho_{1,0}(x) dx \right)
		\end{align*}
		which can be small enough for $R$ close enough to $1$. On the other hand, recall that for any time $t\geq 0$,
		$$
		\rho_R(t)=\frac{1}{R^3} \rho_{1,0}\left(\frac{\cdot}{R}-c \frac{t}{R^2\,  \omega_1}\right).
		$$
		where $\omega_1$ is the volume of the unit ball in $\R^3$.
		Hence, by definition of the first Wasserstein distance we get for $f(x)=x_3 \in \text{Lip}(\R^3)$,
		\begin{align*}
			W_1(\rho_R(t),\rho_1(t)) &\geq  \left| \int x_3 \rho_R(t,x)dx - \int x_3\rho_{1}(t,x) \right|\\
			&\geq \left| \int (T_R(x))_3 \rho_{1,0}(x)dx - \int (T_1(x))_3 \rho_{1,0}(x)dx \right|\\
			&= \left| \int( T_R(x)-T_1(x))_3  \rho_{1,0}(x)dx \right|\\
			&=\left| \int\left( (R-1)x_3+c_3 \frac{t}{\omega_1}\left(\frac{1}{R}-1 \right)  \right)\rho_{1,0}(x)dx \right|\\
			&=\left|(R-1) \left(\int x_3 \rho_{1,0}(x) dx \right)+c_3 \frac{t}{\omega_1}\left(\frac{1}{R}-1   \right) \right|
		\end{align*}
		where we used a change of variable in both integrals with $T_R(x)=Rx+c \frac{t}{R \omega_1}$. The first term above is vanishing by symmetry hence 
		$$
		W_1(\rho_R(t),\rho_1(t)) \geq |c_3| \frac{t}{\omega_1}\left|\frac{1}{R}-1    \right|
		$$
		which can be made arbitrary large by taking $t$ large enough.

	\end{proof}

	\begin{Proposition}
		Let $R>0$.	Any travelling wave type solution $\rho_R$ is unstable for the $L^1$ distance on $\R^3$.
	\end{Proposition}
	
	\begin{proof}
		Again we set $R=1$ for simplicity.
		We will prove the following claim:
		\begin{multline*}
			\forall\, \eta>0, \exists \bar{\rho}_0 \in \mathcal{P}(\R^3)\cap L^\infty(\R^3)\,, \exists T>0 \,  \text{ such that}\\
			\|\bar{\rho}_0-\rho_{1,0}\|_{L^1} <\eta \quad  \text{ and } \quad  \|\bar{\rho}(t)-\rho_1(t)\|_{L^1} \geq 2,\, \forall\, t \geq T
		\end{multline*}
		where $\bar \rho$ is the solution associated to the initial datum $\bar \rho_0$.
		
		The idea is, once again, to exploit the explicit formulas given in Corollary~\ref{Cor:explicit_formula_rho_R}. Let $R>0$ different than $1$ and for any $t\geq 0$, we set
		\[
		I_R(t)=\|\rho_{R}(t) - \rho_{1}(t) \|_{L^1(\R^3)}\, ,
		\]
		the $L^1$ distance between the solutions associated respectively to the initial data $\rho_{R,0}$ and $\rho_{1,0}$, at time $t$.

		We first prove that $I_R(0)$ can be arbitrary small, provided that $R$ is close enough to $1$. We denote by $\omega_1$ the volume of the unit ball in $\R^3$ and by $\omega_R$ the volume of the ball of radius $R$, so that $\omega_R=R^3\omega_1$.
		
		Assume that $R<1$. Then, we have
		\begin{align*}
			I_R(0) & = \int_{\R^3}|\rho_{R,0}(x)-\rho_{1,0}(x)|\, dx \\
			& = \int_{B(0,R)} \left| \frac{1}{\omega_R} - \frac{1}{\omega_1}  \right|\, dx + \int_{B(0,1)\setminus B(0,R)}\frac{1}{\omega_1}\, dx \\
			& = \omega_R \left( \frac{1}{R^3}-1\right) \frac{1}{\omega_1} + (\omega_1-\omega_R)\frac{1}{\omega_1} \\
			& = 2(1-R^3)\, .
		\end{align*}
		By similar computations, for $R>1$, there holds
		\[
		I_R(0) = 2\left(1 - \frac{1}{R^3} \right) \, .
		\]
		Hence, for a given $\eta>0$, there exists $R\neq 1$ such that $I_R(0)<\eta$.
		
		Next, we observe that, if $t$ is large enough, then the support of the solutions $\rho_R(t)$ and $\rho_1(t)$ are disjoint. Indeed, by formula~\eqref{Formula:rho_R(t,x)}, $\rho_R(t)$ is supported in the ball of radius $R$ and center $\alpha_R(t):=\frac{tc}{\omega_1\, R}$. In particular, 
		\[
		|\alpha_R(t)-\alpha_1(t)| = \left| \frac{1}{R} - 1  \right| \frac{t|c|}{\omega_1}
		\]
		Hence, there exists $T_R>0$ such that for any $t\geq T_R$, $|\alpha_R(t)-\alpha_1(t)| > R+1$, which implies that the balls $B(\alpha_R(t), R)$ and $B(\alpha_1(t), 1)$ are disjoint.
		
		As a result, for any $t\geq T_R$, 
		\begin{align*}
			I_R(t) & = \int_{\R^3}|\rho_R(t,x)-\rho_1(t,x)|\, dx \\
			& = \frac{1}{\omega_1} \left( \int_{B(\alpha_R(t), R)} \frac{1}{R^3} 1_{B(\alpha_R(t), R)}(x)\, dx +  \int_{B(\alpha_1(t), 1)}  1_{B(\alpha_1(t), 1)}(x)\, dx   \right) \\
			& = \frac{1}{\omega_1} \left( \frac{\omega_R}{R^3} + \omega_1 \right)\\
			& = 2\, .
		\end{align*}
		Thus, the initial datum $\bar \rho_0=\rho_R$ satisfies $\|\bar \rho_0-\rho_{1,0}\|_{L^1}<\eta$ 
		and $\|\bar \rho(t)-\rho_1(t)\|_{L^1}=2$ for any $t\geq T_R$.
	\end{proof}

	\section{Linearization of the surface evolution model in spherical coordinates}\label{Sec:linearization}
	In this section, we consider $\rho_0$ being the characteristic function of a bounded domain $B_0$ and assume that the initial blob is axisymmetric and can be described by a spherical parametrization
	$$
	B_0=\left\{r_0(\theta) \begin{pmatrix}
		\cos(\phi) \sin(\theta)\\
		\sin(\phi) \sin(\theta)\\
		\cos(\theta)
	\end{pmatrix},\ (\theta,\phi)\in[0,\pi]\times[0,2\pi]\right\}.
	$$
	We denote by $B_t$ the domain at time $t$ such that $\rho(t)=1_{B_t}$. In order to use a spherical parametrization, we set $c(t)=(0,0,c_3(t))$ the position at time $t$ of a reference point, and search for $B_t$ of the form $B_t= c(t)+ \tilde{B}_t$ where
	$$
	\tilde{B}_t=\left\{r(t,\theta) \begin{pmatrix}
		\cos(\phi) \sin(\theta)\\
		\sin(\phi) \sin(\theta)\\
		\cos(\theta)
	\end{pmatrix},\ (\theta,\phi)\in[0,\pi]\times[0,2\pi]\right\}.
	$$
	The velocity of the point $c(t)$ can be chosen arbitrarily, in particular to impose that $c(t)$ is transported along the flow meaning that $\dot{c}= u(t,c)$.
	
	We recall the following result, proven in~\cite{Mecherbet2020}:
	
	\begin{Proposition}\label{prop_model}
		$r$ satisfies the following hyperbolic equation
		\begin{equation}\label{equation_r_hyp}
			\left \{
			\begin{array}{rcl}
				\partial_t r + \partial_\theta r A_1[r] &= &A_2[r], \\
				r(0,\cdot)&=& r_0.
			\end{array}
			\right.
		\end{equation}
		In the case where the reference point $c=(0,0,c_3)$ is transported along the flow \emph{i.e.}\! $u(c)=\dot{c}$, we have $c=c[r]=(0,0,c[r]_3)$ and
		\begin{equation}\label{dotc_3}
			\left\{
			\begin{array}{rcl}
				\dot{c}[r]_3(t)&=& -\frac{1}{4}\displaystyle{\int_0^\pi} r^2(t,\bar\theta) \sin(\bar \theta) \left(1-\frac{1}{2} \sin^2(\bar \theta)\right) d \bar \theta ,\\
				c[r]_3(0)&=&0.
			\end{array}
			\right.
		\end{equation}
		The operators $A_1[r]$ and $A_2[r]$ are defined as follows:
		\begin{multline}\label{eqA1}
			A_1[r](t,\theta) := \\
			-\frac{1}{8\pi r(t,\theta)} \int_0^{2\pi} \int_0^\pi \frac{r(t,\bar  \theta) \sin(\bar  \theta) - \partial_\theta r(t,\bar \theta) \cos(\bar \theta)}{\sqrt{\gamma[r](t,\theta,\bar \theta,\phi)}} r(t,\bar \theta) \sin(\bar \theta) \Big( r(t, \theta) \cos(\phi)\\-r(t,\bar \theta)\Big \{ \cos(\bar \theta) \cos( \theta)\cos(\phi)+\sin (\bar \theta) \sin( \theta)  \Big \} \Big) d \bar \theta d\phi +\frac{\dot{c}_3\sin(\theta)}{r(t,\theta)},
		\end{multline}
		\begin{multline}\label{eqA2}
			A_2[r](t, \theta) := \\
			-\frac{1}{8\pi} \int_0^{2\pi} \int_0^\pi \frac{r(t,\bar \theta) \sin(\bar \theta) - \partial_\theta r(t,\bar\theta) \cos(\bar\theta)}{\sqrt{\gamma[r](t,\theta,\bar \theta,\phi)}} r^2(t,\bar \theta) \sin(\bar\theta)\times \\ \Big(-\sin(\theta)\cos(\bar\theta)  \cos(\phi)+\cos( \theta)\sin(\bar\theta) \Big ) d\bar \theta d\phi- \dot{c}_3 \cos(\theta) \, ,
		\end{multline}
		\begin{equation}\label{eqB}
			\gamma[r](\theta,\bar \theta,\phi)= r^2(\theta)+r^2(\bar \theta)-2 r(\theta) r(\bar \theta)(\sin (\theta) \sin (\bar \theta) \cos(\phi) + \cos(\theta) \cos(\bar \theta) ).
		\end{equation}
	\end{Proposition}
	A local existence and uniqueness result was proven in \cite[Thereom 2.2]{Mecherbet2020} in the following sense: 
	\begin{thm}
		Let $r_0\in {C}^{0,1}[0,\pi]$ such that $|r_0|_*:= \underset{\theta \in [0,\pi]}{\inf} r_0(\theta) >0$. There exist $T>0$ and a unique $r\in {C}(0,T;{C}^{0,1}(0,\pi))$ satisfying the hyperbolic equation \eqref{equation_r_hyp}. Moreover, there exists a unique associated reference point $c=c[r]\in {C}(0,T)$ satisfying \eqref{dotc_3} .
	\end{thm}
	As explained in the reference above, a similar result holds for a reference center $c$ satisfying an arbitrary dynamics provided that we have a uniform bound on $\dot{c}$ and a stability estimate with respect to $r$ if $c=c[r]$.
	In the case where $r_0=1$ with $\dot{c}:=-4/15$ we recover a stationary state for the system corresponding to the travelling wave solution \emph{i.e.}\! $r(t)=1$ with the right center velocity. 
	
	In the following section, we aim to study the linearized model around $r_0=1$ with $c$ given by $\dot{c}:=-\frac{4}{15} e_3$.

	\subsection{Linearized model}
	\begin{Proposition}\label{Prop:DefL}
		The linearized system around $r_0=1$ with $\dot{c}=\dot{c}^\star=-\frac{4}{15}e_3$ reads 
		\begin{equation}\label{eq_linearise}
			\partial_t h + u_0 \partial_\theta h=L[h]
		\end{equation}
		with $u_0(\theta)=-\frac{1}{15}\sin(\theta)$ and $L$ given by
		\begin{equation}\label{formule_L}
			L[h](\theta)= J[h](\theta) + K(\theta)h(\theta),
		\end{equation}
		where
		\begin{equation}\label{Formula:K(theta)}
			K(\theta) = \frac{2}{15}\cos(\theta),  \end{equation}
		\begin{align}
			&J[h](\theta)\notag\\
			&=-\frac{1}{8\pi} \int_{[0,\pi]\times[0,2\pi]} \frac{\sin(\bar \theta) \partial_{\bar \theta} \gamma(\theta,\bar \theta,\phi)}{2\sqrt{\gamma(\theta,\bar \theta,\phi)}}
			\left[\frac{5}{2} h(\bar \theta) \sin(\bar \theta) -h'(\bar \theta) \cos(\bar \theta)\right ] d\bar \theta d\phi\label{formule_J}
		\end{align}
		and $$\gamma(\theta,\bar \theta,\phi)=2-2\sin(\theta) \sin(\bar \theta) \cos(\phi) - 2\cos(\theta) \cos(\bar \theta).$$
		
	\end{Proposition}
	\begin{proof}
		Using the notations from the previous section, the linearized system reads 
		\begin{equation}\label{eq:linearized}
			\partial_t h + A_1[r_0] \partial_\theta h = L[r_0]h
		\end{equation}
		with $r_0=1$. Introducing the notations from~\cite[Appendix 1]{Mecherbet2020}, $A_1[1]$ is given by
		$$A_1[1](\theta)=(\mathcal{U}[1](\theta)-\dot{c}^*)\cdot \partial_\theta e(\theta,0) $$
		where $\dot{c}=-\frac{4}{15}e_3$ and $\mathcal{U}[1]$ corresponds to the solution of Stokes equation with right-hand side $-1_{B(0,1)}e_3$. An explicit formula can be found for instance in the proof of~\cite[Lemma 3.1]{Mecherbet2020} or in~\cite{EMG}
		with $F=-e_3$, $\mu=R_0=1$. We set
		\begin{equation}\label{vitesse_A_1[1]}
			u_0(\theta)=A_1[1](\theta)=- \frac{1}{15} \sin (\theta)
		\end{equation}
		and define $L[1]$ as the linearized of $A_2[r]$ around $r_0=1$, which is given by
		\begin{align*}
			&-8\pi L[1]h\\
			&= \int_{[0,2\pi]\times[0,\pi]} \frac{h(\bar \theta) \sin (\bar \theta) - \partial_\theta h(\bar \theta) \cos (\bar \theta)}{\sqrt{\gamma[1](\bar \theta,\theta,\phi)}} \sin(\bar \theta) \\
			&\quad \times \left[-\sin (\theta) \cos (\bar \theta) \cos (\phi)+ \cos (\theta) \sin (\bar \theta) \right]\\
			&\quad +2\frac{\sin^2(\bar \theta)}{\sqrt{\gamma[1](\bar \theta,\theta,\phi)}}  h(\bar \theta)  \left[-\sin (\theta) \cos (\bar \theta) \cos (\phi)+ \cos (\theta) \sin (\bar \theta) \right]\\
			&\quad - \frac{ h(\bar \theta)+ h(\theta)}{\sqrt{\gamma^3[1](\bar \theta,\theta,\phi)}} \left[-\sin (\theta) \cos (\bar \theta) \cos (\phi)+ \cos (\theta) \sin (\bar \theta) \right]\\
			&\quad \times \sin^2(\bar \theta) \left[1-\sin(\theta) \sin(\bar \theta)\cos( \phi) - \cos(\theta) \cos (\bar \theta) \right]d \bar \theta d \phi
		\end{align*}
		where we used, dropping the dependencies with respect to the variables $\theta$, $\bar \theta$, $\phi$:
		\begin{multline*}
			\gamma[r+h]=\gamma[r]+ 2 r(\theta) h(\theta)+ 2r(\bar \theta)h(\bar \theta)\\-2[r(\theta) h(\bar \theta)+ r(\bar \theta)h(\theta)](\sin(\theta) \sin(\bar \theta)\cos( \phi) + \cos(\theta) \cos (\bar \theta))+o(\|h\|)
		\end{multline*} 
		\begin{align*}
			&\gamma[1+h]^{-1/2}
			\\
			&=\frac{1}{\sqrt{\gamma[1]}}\left(1- \frac{ h(\bar \theta)+ h(\theta)}{\gamma[1]} \Big[1-\sin(\theta) \sin(\bar \theta)\cos( \phi) - \cos(\theta) \cos (\bar \theta)\Big] \right)\\
			&\qquad +o(\|h\|)\\
			&=\frac{1}{\sqrt{\gamma[1]}}- \frac{ h(\bar \theta)+ h(\theta)}{\sqrt{\gamma[1]}^3} \Big[1-\sin(\theta) \sin(\bar \theta)\cos( \phi) - \cos(\theta) \cos (\bar \theta))\Big]\\
			&\qquad +o(\|h\|)
		\end{align*}
		with
		\[
		\gamma[1](\bar \theta, \theta,\phi)={2\left[1-\sin(\theta) \sin(\bar \theta)\cos( \phi) - \cos(\theta) \cos (\bar \theta) \right]}.
		\]
		The above expression of $-8\pi L[1]h$ simplifies to 
		\begin{align*}
			& -8\pi L[1]h\\
			& = \int_{[0,2\pi]\times[0,\pi]} \frac{h(\bar \theta) \sin (\bar \theta) - \partial_\theta h(\bar \theta) \cos (\bar \theta)}{\sqrt{\gamma[1](\bar \theta,\theta,\phi)}} \sin(\bar \theta)\\
			&\quad \times \left[-\sin (\theta) \cos (\bar \theta) \cos (\phi)+ \cos (\theta) \sin (\bar \theta) \right]\\
			&\quad+2\frac{\sin^2(\bar \theta)}{\sqrt{\gamma[1]}}  h(\bar \theta)  \left[-\sin (\theta) \cos (\bar \theta) \cos (\phi)+ \cos (\theta) \sin (\bar \theta) \right]\\
			&\quad- \frac{ h(\bar \theta)+ h(\theta)}{2\sqrt{\gamma[1]}}\sin^2(\bar \theta) \left[-\sin (\theta) \cos (\bar \theta) \cos (\phi)+ \cos (\theta) \sin (\bar \theta) \right] \\
			&=\int_{[0,2\pi]\times[0,\pi]}\frac{\sin (\bar \theta)\left[-\sin (\theta) \cos (\bar \theta) \cos (\phi)+ \cos (\theta) \sin (\bar \theta) \right]}{\sqrt{2\left[1-\sin(\theta) \sin(\bar \theta)\cos( \phi) - \cos(\theta) \cos (\bar \theta) \right]}} \\
			&\quad \times  \left(h(\bar \theta) \sin (\bar \theta) - \partial_\theta h(\bar \theta) \cos (\bar \theta) +2 h(\bar \theta)\sin (\bar \theta)-\frac{h(\bar \theta) + h(\theta)}{2} \sin( \bar \theta) \right) \\
			&=\int_{[0,2\pi]\times[0,\pi]}\frac{\sin (\bar \theta)\left[-\sin (\theta) \cos (\bar \theta) \cos (\phi)+ \cos (\theta) \sin (\bar \theta) \right]}{\sqrt{2} \sqrt{1-\sin(\theta) \sin(\bar \theta)\cos( \phi) - \cos(\theta) \cos (\bar \theta) }} \\
			&\quad\times  \left( \frac{5}{2} h(\bar \theta) \sin (\bar \theta) - \partial_\theta h(\bar \theta) \cos (\bar \theta) -\frac{ h(\theta)}{2} \sin( \bar \theta) \right) d \bar \theta d \phi\, .
		\end{align*}
		Setting $L[1]h:=L[h]$, this shows formula \eqref{formule_L} with $J$ as in \eqref{formule_J} and $K$ corresponding to
		\[
		K(\theta)=\frac{1}{16\pi} \int_{[0,\pi]\times[0,2\pi]} \frac{\sin^2(\bar \theta) \partial_{\bar \theta} \gamma(\theta,\bar \theta,\phi)}{2\sqrt{\gamma(\theta,\bar \theta, \phi)}} d\bar \theta d\phi.
		\]
		
		Let us finally prove formula~\eqref{Formula:K(theta)}. Noticing that $\frac{\partial_{\bar \theta} \gamma(\theta,\bar \theta,\phi)}{2\sqrt{\gamma(\theta,\bar \theta,\phi)}}$ is the derivative of the mapping $\bar \theta \mapsto \sqrt{\gamma(\theta,\bar \theta,\phi)}$ and using integration by parts, we can rephrase the previous expression of $K(\theta)$ as
		\[
		K(\theta) = -\frac{1}{8\pi} \int_{[0,\pi]\times[0,2\pi]} \sqrt{\gamma(\theta,\bar \theta,\phi)}  \sin(\bar\theta) \cos(\bar\theta) d\bar \theta d\phi\, .
		\]
		Setting 
		\[
		x=\begin{pmatrix}
			\sin(\theta) \\ 0 \\ \cos(\theta)
		\end{pmatrix} , \quad 
		y = \begin{pmatrix}
			\sin(\bar \theta) \cos(\phi) \\ \sin(\bar \theta) \sin(\phi) \\ \cos(\bar \theta)
		\end{pmatrix},
		\]
		we see that $|x-y| = \sqrt{\gamma(\theta,\bar \theta,\phi)}$, so that the above expression of $K(\theta)$ can be interpreted using a surface integral on the unit sphere $\mathbb{S}^2$ in $\R^3$, as
		\[
		K(\theta) =  -\frac{1}{8\pi} \int_{\mathbb{S}^2} y_3 |x-y| \, d\sigma(y),
		\]
		where $\sigma$ stands for the surface measure on the sphere. Now, introducing the rotation matrix
		\[
		R = \begin{pmatrix} \cos(\theta) & 0 & \sin(\theta) \\ 0 & 1 & 0 \\  -\sin(\theta) & 0 & \cos(\theta)   \end{pmatrix},
		\]
		the change of variable $y=R z$ and observing that $x=R e_3$, we obtain
		\begin{align*}
			K(\theta) & = -\frac{1}{8\pi} \int_{\mathbb{S}^2} (R z)_3 |R (e_3-z)| \, d\sigma(R z) \\
			& = -\frac{1}{8\pi} \int_{\mathbb{S}^2} (R z)_3 |e_3-z| \, d\sigma(z) \\ 
			& = \frac{1}{8\pi}\sin(\theta) \int_{\mathbb{S}^2} z_1 |e_3-z| \, d\sigma(z) -  \frac{1}{8\pi}\cos(\theta) \int_{\mathbb{S}^2} z_3 |e_3-z| \, d\sigma(z).
		\end{align*}
		By a symmetry argument, it is easy to see that the integral $\int_{\mathbb{S}^2} z_1 |e_3-z| \, d\sigma(z)$ vanishes. As regards the second integral $ \int_{\mathbb{S}^2} z_3 |e_3-z| \, d\sigma(z)$, it can be computed explicitly, coming back to spherical coordinates. 
		%and setting 
		%	\[
		%	\omega = \begin{pmatrix}
			%	\sin(\bar \theta) \cos(\phi) \\ \sin(\bar \theta) \sin(\phi) \\ \cos(\bar \theta)
			%	\end{pmatrix}\, .
		%	\]
		%	This yields
		%	\begin{align*}
			%	\int_{\mathbb{S}^2} \omega_3 |e_3-\omega| \, d\sigma(\omega) & = \int_{0}^\pi \int_0^{2\pi} \cos(\bar \theta) \sqrt{2} \sqrt{1-\cos(\bar{\theta})} \sin(\bar \theta)\, d\phi d\bar \theta \\
			%	& = 2\pi\sqrt{2} \int_{0}^\pi \cos(\bar \theta) \sqrt{1-\cos(\bar{\theta})} \sin(\bar \theta)\,  d\bar \theta\, .
			%	\end{align*}
		%	Using the change of variable $\bar \theta = \arccos(u)$ and integrating by parts, last integral reduces to
		%	\begin{align*}
			%	\int_{0}^\pi \cos(\bar \theta) \sqrt{1-\cos(\bar{\theta})} \sin(\bar \theta)\,  d\bar \theta\, & = \int_{-1}^1 u\sqrt{1-u}\, du \\
			%	& = \frac{2}{3}\left[ -\frac{2}{5}(1-u)^{5/2} - u(1-u)^{3/2} \right]_{-1}^1
			%	\\
			%	& = -\frac{4\sqrt{2}}{15}\, .
			%	\end{align*}
		One obtains $\int_{\mathbb{S}^2} z_3 |e_3-z| \, d\sigma(z) = -\frac{16\pi}{15}$, which gives expression~\eqref{Formula:K(theta)}.
		
	\end{proof}
	
	\subsection{Well-posedness of the linearized model}
	We focus in the following Proposition solely on the non local operator $J$ appearing in the definition of $L$ (see formula~\eqref{formule_L}), since the second part is simply a multiplication by a smooth function $K$ (formula~\eqref{Formula:K(theta)}).

	\begin{Proposition}\label{well_posedness_Lh} \phantom{a}
		\begin{enumerate}[label=\roman*.]
			\item The operator $J$ is continuous from $C^1([0,\pi])$ to $C^1([0,\pi])$, and from $W^{1,p}(0,\pi)$ to $W^{1,p}(0,\pi)$ for $p>2$. \label{item1}
			\item  Let $p>2$. The operator $J$ can be defined as a continuous linear operator from $L^p(0,\pi)$ into $L^\infty(0,\pi)$ (resp. from $C([0,\pi])$ into itself) as follows: \label{item2}
			\begin{align*}
				J[h](\theta)&=-\frac{1}{8\pi} \int_{[0,\pi]\times[0,2\pi]} \Bigg[\frac{\partial_{\bar \theta} \gamma(\theta,\bar \theta,\phi)}{2\sqrt{\gamma(\theta,\bar \theta,\phi)}}
				\left(1+\frac{1}{2}\sin^2(\bar \theta) \right) 
				\\
				&+\sin(\bar \theta)\cos(\bar \theta)\left( \frac{ \partial^2_{\bar \theta} \gamma(\theta,\bar \theta,\phi)}{2\sqrt{\gamma(\theta,\bar \theta,\phi)}}
				-\frac{1}{2}\frac{\left( \partial_{\bar \theta} \gamma(\theta,\bar \theta,\phi)\right)^2 }{2\sqrt{\gamma(\theta,\bar \theta,\phi)}^3}  \right)
				\Bigg] h(\bar \theta)  d\bar \theta d\phi\, .
			\end{align*}
			Moreover, we have the following explicit bound for $\|J\|$, the operator norm of $J$ from $L^p$ to $L^\infty$:
			\begin{equation}\label{eq:borne_norm_J_infty}
				\|J\|\leq     \frac{1}{8} \left(7  \pi^{1/p'}+2   \left (\int_0^\pi \frac{\sin \bar \theta}{\sqrt{1-\cos \bar \theta}^{p'}} d \bar \theta \right)^{1/p'}  \right)\, ,
			\end{equation}
   where $p'$ is the conjugate exponent of $p$.
		\end{enumerate}
	\end{Proposition}
	\begin{cor}\label{coro1}
		For any $h_0\in C^1([0,\pi])$ (resp.\! $h_0\in W^{1,p}(0,\pi)$, $h_0\in L^p(0,\pi)$, $p>2$  or $h_0 \in C([0,\pi])$)  there exists a unique global solution to the linearized system \eqref{eq:linearized} in $C([0,+\infty[, C^1([0,\pi])$ (resp. in $C([0,+\infty[, W^{1,p}(0,\pi))$, $C([0,+\infty[, L^p(0,\pi))$ or $C([0,+\infty[, C([0,\pi])$).	
		In all cases we have
		$$
		h(t,\theta)=h_0(\Theta(0,t,\theta))+\int_0^t Lh(s,\Theta(s,t,\theta)) ds
		$$
		where $\Theta \in C^1([0,+\infty[ \times [0,\pi])$ is the characteristic flow associated to $A_1[1]$ (see \eqref{vitesse_A_1[1]}).
	\end{cor}
	\begin{proof}[Proof of Proposition \ref{well_posedness_Lh}]
		\textbf{Proof of item \ref{item1}.}
		First note that one can rewrite the factor appearing in the integral from formula~\eqref{formule_J} as
		\begin{multline}\label{formule_desingularisante}
			\frac{-\sin(\theta)\cos(\bar \theta) \cos(\phi)+\cos(\theta) \sin(\bar \theta)}{\sqrt{2-2\sin(\theta)\sin(\bar \theta) \cos(\phi)-{2}\cos(\theta) \cos(\bar \theta)}}\\=\frac{-(\sin(\theta) \cos(\phi)-\sin(\bar \theta))\cos(\bar \theta)+(\cos(\theta)-\cos(\bar \theta)) \sin(\bar \theta)}{\sqrt{(\sin (\theta) \cos(\phi)-\sin (\bar \theta))^2+\sin^2(\phi) \sin^2(\theta)+ (\cos (\theta)-\cos(\bar \theta))^2}}.
		\end{multline}
		This shows that this term is not singular and bounded by a constant, hence 
		\begin{equation}\label{eq:bound_L}
			\|J[h]\|_\infty \leq C (\|h\|_\infty+\|h'\|_\infty)
		\end{equation}
		and for any $p\geq 1$,
		\begin{equation}\label{eq:bound_L_bis}
			\|J[h]\|_{L^p} \leq C (\|h\|_{L^p}+\|h'\|_{L^p}).
		\end{equation}
		
		For the derivative, the only term that requires a further explanation is the one coming from deriving the above term with respect to $\theta$. The result can be decomposed as
		\begin{align*}
			I_1+I_2 = &\bigg|\frac{-\cos(\theta)\cos(\bar \theta) \cos(\phi)-\sin(\theta) \sin(\bar \theta)}{\sqrt{2-2\sin(\theta)\sin(\bar \theta) \cos(\phi)-{2}\cos(\theta) \cos(\bar \theta)}} \bigg|\\
			&+\bigg|\frac{-\sin(\theta)\cos(\bar \theta) \cos(\phi)+\cos(\theta) \sin(\bar \theta)}{\sqrt{2-2\sin(\theta)\sin(\bar \theta) \cos(\phi)-{ 2}\cos(\theta) \cos(\bar \theta)}^3} \\
			& \qquad \times(-\cos(\theta)\sin(\bar\theta)\cos(\phi)+\sin(\theta)\cos(\bar\theta))\bigg|.
		\end{align*}
		When multiplying the first term by $\sin (\bar \theta)$ appearing in the definition of $L$, we get an integrable function since, by Lemma~\ref{lemme_int_S2}, the mapping  
		$$(\bar \theta,\phi)\mapsto \frac{\sin (\bar \theta)}{\sqrt{2-2\sin(\theta)\sin(\bar \theta) \cos(\phi)-2\cos(\theta) \cos(\bar \theta)}}$$
		is integrable for any $\theta \in[0,\pi]$. For $I_2$ we use again that 
		\begin{align*}
			&-\cos(\theta)\sin(\bar\theta)\cos(\phi)+\sin(\theta)\cos(\bar\theta)\\
			& =-\cos(\theta)(\sin(\bar \theta)\cos (\phi) - \sin(\theta))+\sin(\theta)(\cos(\bar \theta)-\cos(\theta))
		\end{align*}
		together with 
		\begin{multline*}
			\sqrt{2-2\sin(\theta)\sin(\bar \theta) \cos(\phi)-2\cos(\theta) \cos(\bar \theta)}\\= \sqrt{(\sin(\bar \theta)\cos(\phi)-\sin(\theta))^2+ \sin^2(\phi)\sin^2(\bar \theta)+(\cos(\theta)-\cos(\bar \theta))^2}
		\end{multline*}
		to simplify the singularity in $I_2$ as follows:
		$$
		I_2\leq \frac{C}{\sqrt{2-2\sin(\theta)\sin(\bar \theta) \cos(\phi)-2\cos(\theta) \cos(\bar \theta)}},
		$$
		and conclude again using Lemma~\ref{lemme_int_S2}. 
		As a result, the following estimate holds true: 
		\begin{equation}\label{eq:bound_L'}
			\|\partial_\theta J[h]\|_\infty\leq C (\|h\|_\infty+\|h'\|_\infty).
		\end{equation}
		
		To derive estimates in $L^p$-norm, we first recall that thanks to \eqref{formule_desingularisante}, there is no singularity in the integral hence it is straightforward that $\|J[h]\|_{L^p} \leq C \|h\|_{W^{1,p}}$. As regards the derivative, notice first that by H\"older inequality,
		\begin{align*}
			&\int_0^\pi \Big| \int_{[0,\pi]\times[0,2\pi]}\sin(\bar \theta)(I_1+I_2)(|h(\bar \theta)|+|h'(\bar \theta)|) d \bar \theta d \phi\Big|^p d \theta\\
			&\leq C_p\int_0^\pi \Big| \int_{[0,\pi]\times[0,2\pi]}\frac{\sin(\bar \theta) }{\sqrt{\gamma(\bar \theta,\phi,\theta)}}(|h(\bar \theta)|+|h'(\bar \theta)|)d \bar \theta d \phi \Big|^p d \theta\\
			&\leq C_p\int_0^\pi \Big(\int_{[0,\pi]\times[0,2\pi]}\frac{\sin(\bar \theta) d \bar \theta d \phi}{\sqrt{\gamma(\bar \theta,\phi,\theta)}^{p'}} \Big)^{p/{p'}} \left(\int_0^\pi | h(\bar \theta)|^p+|h'(\bar \theta)|^p  d \bar\theta\right) .
		\end{align*}
		We conclude, using Lemma~\ref{lemme_int_S2} for $\alpha=p'<2$, that for any $p>2$,
		\begin{equation}\label{eq:bound_L'_bis}
			\|\partial_\theta J[h]\|_{L^p}\leq C (\|h\|_{L^p}+\|h'\|_{L^p}).
		\end{equation}
		\textbf{Proof of item \ref{item2}} Let us prove the continuity of $J$ in the $L^p(0,\pi)$ setting, the result in the $C([0,\pi])$ setting holds analogously.
		Recall the definition of $J[h]$:
		\begin{align*}
			& J[h](\theta)\\
			& =-\frac{1}{8\pi} \int_{[0,\pi]\times[0,2\pi]} \frac{\sin(\bar \theta) \partial_{\bar \theta} \gamma(\theta,\bar \theta,\phi)}{2\sqrt{\gamma(\theta,\bar \theta,\phi)}}
			\left[\frac{5}{2} h(\bar \theta) \sin(\bar \theta) -h'(\bar \theta) \cos(\bar \theta)\right ] d\bar \theta d\phi\, ,
		\end{align*}
		with $$\gamma(\theta,\bar \theta,\phi)=2-2\sin(\theta) \sin(\bar \theta) \cos(\phi) - 2\cos(\theta) \cos(\bar \theta).$$
		Using an integration by parts, we have 
		\begin{align*}
			&-\int_{[0,\pi]\times[0,2\pi]} \frac{\sin(\bar \theta) \partial_{\bar \theta} \gamma(\theta,\bar \theta,\phi)}{2\sqrt{\gamma(\theta,\bar \theta,\phi)}}
			h'(\bar \theta) \cos(\bar \theta) d\bar \theta d\phi  \\
			&= \int_{[0,\pi]\times[0,2\pi]}(\cos^2(\bar \theta) -\sin^2(\bar \theta) ) \frac{\partial_{\bar \theta} \gamma(\theta,\bar \theta,\phi)}{2\sqrt{\gamma(\theta,\bar \theta,\phi)}}
			h(\bar \theta) d\bar \theta d\phi\\
			&\quad + \int_{[0,\pi]\times[0,2\pi]}\sin(\bar \theta)\cos(\bar \theta)\left( \frac{ \partial^2_{\bar \theta} \gamma(\theta,\bar \theta,\phi)}{2\sqrt{\gamma(\theta,\bar \theta,\phi)}}
			-\frac{1}{2}\frac{\left( \partial_{\bar \theta} \gamma(\theta,\bar \theta,\phi)\right)^2 }{2\sqrt{\gamma(\theta,\bar \theta,\phi)}^3}  \right)h(\bar \theta) d\bar \theta d\phi\, .
		\end{align*}
		which yields the desired formula. We recall  that by \eqref{formule_desingularisante} 
		$$
		\left|\frac{\partial_{\bar \theta} \gamma(\theta,\bar \theta,\phi)}{2\sqrt{\gamma(\theta,\bar \theta,\phi)}}\right|\leq 1
		$$
		hence
		$$
		\left| \frac{ \partial^2_{\bar \theta} \gamma(\theta,\bar \theta,\phi)}{2\sqrt{\gamma(\theta,\bar \theta,\phi)}}
		-\frac{1}{2}\frac{\left( \partial_{\bar \theta} \gamma(\theta,\bar \theta,\phi)\right)^2 }{2\sqrt{\gamma(\theta,\bar \theta,\phi)}^3}  \right|\leq \frac{2}{\sqrt{\gamma(\theta,\bar \theta,\phi)}}\, .
		$$
		This ensures that $\displaystyle J[h](\theta)=\int_{[0,\pi]\times [0,2\pi]} h(\bar \theta) F(\bar \theta, \theta,\phi)d \bar \theta d \phi$ where for all $\theta, \bar \theta \in (0,\pi)$ and $\phi \in(0,2\pi)$, 
		$$
		\left|F(\bar \theta, \theta,\phi) \right| \leq \frac{1}{8\pi}\left(\frac{7}{2}+\frac{2\sin(\bar \theta)}{\sqrt{\gamma(\theta,\bar \theta,\phi)}}\right)\, .$$
		We conclude by H\"older inequality since,  using an estimate appearing in the proof of Lemma~\ref{lemme_int_S2} for $p>2$, we have 
		\begin{align*}
			%\| F(\theta,\cdot,\cdot)\|_{L^{p'}(0,\pi)\times(0,2\pi)} 
			& \|J[h]\|_\infty \\
			& \leq \frac{1}{8\pi} \left(7 \pi \pi^{1/p'}\|h\|_p +2 \bigg(\int_{(0,\pi)\times(0,2\pi)}\frac{\sin(\bar \theta)}{\sqrt{\gamma(\theta,\bar \theta,\phi)}^{p'}}d \bar \theta d \phi   \bigg)^{1/p'} (2\pi)^{1/p}\|h\|_p   \right)\\
			&\leq \frac{1}{8\pi} \left(7 \pi \pi^{1/p'}\|h\|_p +2 \pi  \bigg (\int_0^\pi \frac{\sin \bar \theta}{\sqrt{1-\cos \bar \theta}^{p'}} d \bar \theta \bigg)^{1/p'}\|h\|_p   \right)\, ,
		\end{align*}
		%\left (\int_0^\pi \frac{\sin \bar \theta}{\sqrt{1-\cos \bar \theta}^{p'}} d \bar \theta \right)^{1/p'}
		which ensures actually that
		$$
		\|J[h]\|_\infty \leq C \|h\|_p\, .
		$$
		with a constant as in \eqref{eq:borne_norm_J_infty}.
	\end{proof}
	\begin{proof}[Proof of Corollary \ref{coro1}]
		Denoting by $\Theta$ the characteristic flow associated to the velocity $A_1[1](\theta)=-\frac{1}{15}\sin(\theta)$, we get that $h$ satisfies the  following fixed point problem
		$$
		h(t,\theta)=h_0(\Theta(0,t,\theta))+\int_0^t Lh(s,\Theta(s,t,\theta))\, ds.
		$$
		Since $L$ is linear continuous, the operator $\mathcal{I}:C([0,t],C^1) \to C([0,t],C^1)$ defined by $$ h \mapsto\left[ \theta \mapsto \int_0^t Lh(s,\Theta(s,t,\theta))\, ds \right]$$ 
		is also linear continuous, and contractant for $t$ small enough. This ensures the local existence and uniqueness of the solution to the linearized system when $h_0\in C^1([0,\pi])$. 
		
		In the  $W^{1,p}$ setting, we also have that $\mathcal{I}: C([0,t],W^{1,p}) \to C([0,t], W^{1,p})$ is linear continuous. Indeed, direct computations using \eqref{vitesse_A_1[1]} show that 
		\begin{equation}\label{Theta}
			\Theta(t,s,\theta)=2\arctan\left( \tan\left(\frac{\theta}{2}\right)e^{\frac{-(t-s)}{15}} \right),
		\end{equation}
		
		\begin{equation}\label{Theta_derivee}
			\partial_\theta \Theta(t,s,\theta)=e^{-\frac{t-s}{15}} \frac{1+\tan^2(y/2)e^{\frac{2(t-s)}{15}}}{\tan^2(y/2)+1},\quad y=\Theta(t,s,\theta)
		\end{equation}
		and then for $s \leq t$, $ |\partial_{\theta}\Theta(s,t,\theta)|\leq  e^{(t-s)/15}$.
		Hence, for any $f \in L^p$, we have $ \|f\circ \Theta(s,t,\cdot)\|_p \leq  \|f\|_{L^p} e^{t/{15p}} $ which yields using \eqref{eq:bound_L_bis}
		$$
		\|\mathcal{I}h (t)\|_{L^p} \leq  t e^{\frac{t}{15p}} \underset{0 \leq s \leq t}{\sup}\|Lh(s)\|_{L^p}\leq C t e^{\frac{t}{15p}} \underset{0 \leq s \leq t}{\sup}\|h(s)\|_{W^{1,p}}.
		$$
		Analogously, thanks to \eqref{eq:bound_L'_bis},
		$$
		\|\partial_\theta \mathcal{I}h (t)\|_{L^p} \leq t e^{\frac{t}{15}\left(1+\frac{1}{p} \right)} \underset{0 \leq s \leq t}{\sup}\|\partial_\theta Lh(s)\|_{L^p}\leq Ct e^{\frac{t}{15}\left(1+\frac{1}{p} \right)} \underset{0 \leq s \leq t}{\sup}\|h(s)\|_{W^{1,p}}.
		$$
		This ensures that $\mathcal{I}: C([0,t],W^{1,p}) \to C([0,t],W^{1,p})$ is contractant for $t$ small enough, which implies local existence and uniqueness. The proof is analogous in the $L^p$ setting.
		
		Eventually, in all cases a standard Gronwall argument ensures that the solution is global.
	\end{proof}

	\subsection{Investigation of the instability of the linearized system}
	
	In the rest of this section we set $A$ the operator formally defined by
	$$
	A= \frac{1}{15}\sin(\theta) \partial_\theta + L \, ,
	$$
	so that the linearized system~\eqref{eq_linearise} reads $\partial_t h = Ah$. As a result, the existence of an eigenvalue $\lambda$ for the operator $A$ with positive real part ensures the instability of system~\eqref{eq_linearise}. Note however that 	$$
	J[h](\pi- \theta) = - J[\tilde{h}](\theta)
, \quad 
	K(\theta)=-K(\pi - \theta)
	$$
where $\tilde{h}(\theta)=h(\pi-\theta)$. This yields
 \begin{Proposition}\label{prop_symetrie_vp}
	If $\lambda \in \C$ is an eigenvalue for the operator $A$ with $h:[0,\pi]\to \C $ its associated eigenvector, then $-\lambda$ is an eigenvalue for the eigenvector $\tilde{h}:\theta\mapsto h(\pi-\theta) $.
\end{Proposition}
Hence, in order to prove instability, it suffices to show existence of a non vanishing eigenvalue.

 As a consequence of Proposition~\ref{Prop:DefL}, the operator $L$ is continuous from $L^p(0,\pi)$ into itself. However, due to the presence of the advective term $\frac{1}{15}\sin(\theta)\partial_\theta$, the differential operator $A$ is unbounded over $L^p(0,\pi)$. This makes the analysis of the eigenvalue problem associated with $A$ and their numerical approximation difficult. 
	
	In order to circumvent this issue, we rely on a strategy that was proposed by Mika in~\cite{Mika}, in the context of neutronic applications of the linear transport theory. The idea consists in turning an eigenvalue problem involving a transport operator, into a fixed point problem that is built using the resolvent of the advective part of the original operator. Adapted to the present context, this procedure leads to rephrasing the eigenvalue problem associated to $A$ as a fixed point problem for a non local operator. The precise statements are given in the next subsection.
	
	\subsubsection{Reformulation of the eigenvalue problem associated with the operator $A$}
	
	\begin{Proposition}\label{prop_T_lamda}
		Let $p>2$ and $\lambda$ such that $\Re(\lambda)>\frac{1}{15 p }$. The operator $T^\lambda$ defined as  \label{item1_eq_vp}
		\begin{equation}\label{def:tlambda}
			T^\lambda[h](\theta):=\int_{-\infty}^{0}Lh(\Theta(\tau,0,\theta)) e^{\lambda \tau }  d \tau\, ,\quad \theta\in (0,\pi)
		\end{equation}
		is continuous from $L^p(0,\pi)$ to $L^p(0,\pi)$.
	\end{Proposition}
	\begin{remark} \label{remarque}
\begin{enumerate}[label=\roman*.]
    \item  A similar version of this result holds in the $C([0,\pi])$ setting with the assumption $\Re(\lambda) >0$. 
    \item An analogous operator can be constructed in order to search for eigenvalues with negative real part \label{remarque2}
    $$
    G^\lambda[h](\theta):=-\int_0^{+\infty}Lh(\Theta(\tau,0,\theta))e^{\lambda \tau} d \tau
    $$
    for $\Re(\lambda)<0$. However the obtained results are symmetric since a straightforward computation using Proposition \ref{prop_symetrie_vp} together with 
    $$
    \Theta(-t,0,\theta)=\pi-\Theta(t,0,\pi-\theta)
    $$
    yields
    $$
    G[h]^\lambda(\theta)=T^{-\lambda}[\tilde{h}](\pi- \theta)
    $$
    with $\tilde{h}(\theta)=h(\pi - \theta)$. This ensures that the eigenvalues of $T^\lambda$ and $G^\lambda$ are the same (with eigenvectors that differ by symmetry $\sim$).
    \end{enumerate}   
\end{remark}

	%         If $\Re(\lambda)>0$ is an eigenvalue associated to an eigenvector $h$ for the operator $Ah= \frac{1}{15}\sin (\theta) \partial_\theta h + Lh$ we have 
	%  $$\forall \theta\in (0,\pi)\quad h(\theta)=
	% \int_{-\infty}^{0}Lh(\Theta(\tau,0,\theta)) e^{\lambda \tau }  d \tau:= T^\lambda[h](\theta)\, .
	% $$
	% Moreover, $ T^\lambda $ is continuous from $L^p \to L^p$. {\color{red}Analogously, it is continuous from $W^{1,p}$ to $W^{1,p}$?}

	\begin{proof}
		Without loss of generality, let assume that $\lambda \in \R$. The same proof can be adapted to handle the case of a complex eigenvalue.
		
		Let us first prove that $T^\lambda[h]$ is well defined for any $h\in L^p(0,\pi)$. 
		Let $t, s \leq 0$ and fix $q\geq 1$.  
		Using H\"older inequality, we get the estimate
		\begin{align*}
			& \int_0^\pi \left| \int_{t}^{s} Lh(\Theta(\tau,0,\theta)) e^{\lambda \tau\, d\tau } \right|^p d \theta  \\
			&\leq   \left|\int_0^\pi \int_t^{s}|Lh(\Theta(\tau,0,\theta))|^p  e^{ \lambda \tau \frac{p}{q}}d \tau d \theta  \right| \left|\int_{t}^{s} e^{\lambda \tau\frac{p'}{q'}} d \tau \right|^{p-1} \\
			&=  \left| \int_t^{s}  e^{ \lambda \tau \frac{p}{q}} \int_0^\pi |Lh(\Theta(\tau,0,\theta))|^p  d \theta d \tau  \right| \left|\int_{t}^{s} e^{\lambda \tau\frac{p'}{q'}} d \tau \right|^{p-1}\, ,
		\end{align*}
		where $p',q'$ stand respectively for the conjugate exponents of $p,q$. 
		Introducing the change of variable  $z=\Theta(\tau,0,\theta)$ in the above formula, we have $\theta=\Theta(0,\tau,z)=2\arctan(\tan(z/2)e^{\tau/15})$, hence
		\begin{equation}\label{eq:def_G}
			d \theta= e^{\tau/15}\frac{1+\tan^2(z/2)}{1+\tan^2(z/2)e^{2\tau/15}} dz=G(\tau,z) dz\, .
		\end{equation}
		In particular, for all $\tau<0$ and $z\in (0,\pi)$,
		\begin{equation*}
			|G(\tau,z)|= \left| e^{-\tau/15} \frac{1+\tan^2(z/2)}{\tan^2(z/2)+e^{-2\tau/15}}\right|
			\leq e^{-\tau/15}\, .
		\end{equation*}
		We deduce that for any $t,s\leq 0$,
		\begin{align}
			& \int_0^\pi \left| \int_{t}^{s} Lh(\Theta(\tau,0,\theta)) e^{\lambda \tau }\, d \tau \right|^p d \theta \notag \\
			&\leq       \left|q'\frac{e^{\lambda s \frac{p'}{q'}}-e^{\lambda t \frac{p'}{q'}}}{\lambda p'} \right|^{p-1} \left|\int_t^{s}e^{\frac{p}{q} \lambda \tau } \int_0^\pi  \left|Lh(z) \right|^p G(\tau,z) d \theta d \tau\right|  \notag\\
			&\leq \Big(\frac{2q'}{\lambda p'}\Big)^{p-1}\left| \int_{t}^{s} e^{\frac{p}{q} \lambda \tau-  \frac{\tau}{15}  } \int_0^\pi  \left|Lh(z) \right|^p d \theta d \tau \right|\notag\\
			&=\Big(\frac{2q'}{\lambda p'}\Big)^{p-1}  \left\|Lh \right\|_p^p    \frac{\left|e^{(\frac{p}{q} \lambda - \frac{1}{15})s}-e^{(\frac{p}{q} \lambda -\frac{1}{15})t }\right|}{\frac{p}{q} \lambda - \frac{1}{15} } \label{LpestimateIntegraletas}
		\end{align}
		where we have chosen $q\geq 1$ such that $ \frac{p}{q} \lambda - \frac{1}{15}>0$. The existence of such a $q$ results from the assumption $\lambda>\frac{1}{15 p}$.

		Thus, if we set for $t<0$
		\begin{equation}\label{Def:F}
			F(t,\theta) =\int_t^0 Lh(\Theta(\tau,0,\theta)) e^{\lambda \tau } d \tau\, , 
		\end{equation}
		inequality~\eqref{LpestimateIntegraletas} with $s=0$ and $\frac{p}{q} \lambda - \frac{1}{15} >0$ shows that 
		$$
		\|F(t,\cdot)\|_p \leq  C_{q,p,\lambda}\left\|Lh \right\|_p    \left(\frac{1-e^{(\frac{p}{q} \lambda -\frac{1}{15})t}}{\frac{p}{q} \lambda - \frac{1}{15}} \right)^{1/p}\, ,
		$$
		with $C_{q,p,\lambda}=\big(\frac{2q'}{\lambda p'} \big)^{1/p'}$.
		Hence $F(t,\cdot) \in L^p$ for any $t<0$ and moreover, for any sequence $t_n \to -\infty$, $F(t_n,\cdot)$ is a Cauchy sequence in $L^p$ thanks to \eqref{LpestimateIntegraletas} and therefore has a limit when $n \to \infty$, which is independent on the choice of the sequence $(t_n)_n$. This means that we can take $t \to - \infty$ in $F(t,\cdot)$ and set 
		$$
		T^\lambda[h](\cdot):= \underset{n \to \infty}{\lim} F(t_n,\cdot) \text { in } L^p(0,\pi)\, ,
		$$
		which satisfies 
		$$
		\|T^\lambda[h]\|_p \leq  C_{q,p,\lambda}\left\|Lh \right\|_p    \left(\frac{1}{\frac{p}{q} \lambda - \frac{1}{15}} \right)^{1/p}\, .
		$$
		By Proposition~\eqref{well_posedness_Lh} and definition~\eqref{formule_L}, $L$ is a continuous operator from $L^p(0,\pi)$ to itself, so the above estimate implies the continuity of $T^{\lambda}$ and concludes the proof of Proposition~\eqref{prop_T_lamda}. 
	\end{proof}

	The introduction of operator $T^\lambda$ defined by~\eqref{def:tlambda} is motivated by the following proposition.
	
	\begin{Proposition}\label{prop_carac_fonction_propre_A}
		Let $p>2$ and $\lambda$ such that $\Re(\lambda)>\frac{1}{15 p }$. A function $h \in L^p(0,\pi)$ is an an eigenfunction for the operator $A$ associated to the eigenvalue $\lambda$ if and only if 
		\begin{equation}\label{FixedPointTlambda}
			T^\lambda[h]=h\, .
		\end{equation}
	\end{Proposition}

	\begin{proof}
		Since the proof is rather lengthy and computational, we postpone it to Appendix~\ref{Appendix:proofPropA}. However, we give here the formal computations that lead to the introduction of operator $T^\lambda$ and to the fixed point problem~\eqref{FixedPointTlambda}.
	
		Let $\lambda$ and $h$ an eigenvalue and its associated eigenvector for the operator $A$. By definition of $L$ (see Proposition~\ref{Prop:DefL}), one has
		\begin{equation}\label{Def:h_vp_A}
			- \frac{1}{15}\sin (\theta) h'(\theta)  + \lambda h(\theta) =Lh (\theta)\, .
		\end{equation}
		Evaluating the previous relation in $\Theta(t,s,\theta)$ yields
		$$
		- \frac{1}{15}\sin (\Theta(t,s,\theta)) h'(\Theta(t,s,\theta))  + \lambda h(\Theta(t,s,\theta)) =L h (\Theta(t,s,\theta))
		$$
		which reads equivalenty
		\begin{equation*}
			\frac{d}{dt} \left (e^{\displaystyle \lambda t  }  h(\Theta(t,s,\theta))   \right)  =Lh (\Theta(t,s,\theta))e^{\displaystyle \lambda t }\, .
		\end{equation*}
		Hence, integrating between $s$ and $t$, we get
		$$
		e^{\displaystyle \lambda t }  h(\Theta(t,s,\theta)) \notag\\
		= e^{\displaystyle \lambda s }  h(\theta)  + \int_s^t  Lh (\Theta(\tau,s,\theta))e^{\displaystyle \lambda \tau  } d \tau\,
		$$
		which yields formally the desired formula when taking $s=0$ and $t \to - \infty$ for $\Re(\lambda)>0$ in a $L^\infty$ setting. We refer to Proposition~\ref{prop_T_lamda} for a rigorous justification of the passage in the limit $t \to - \infty$ in an $L^p$ setting.
	\end{proof}

	% {\color{blue}Hence when considering the weak formulation for for some test function $\psi$ we get 
		% \begin{align*}
			% &\int_0^\pi \int_{-\infty}^{0}Lh(\Theta(\tau,0,\theta)) e^{\lambda \tau }  d \tau \psi(\theta)d \tau  d \theta =\int_{-\infty}^{0}  \int_0^\pi Lh(\Theta(\tau,0,\theta)) e^{\lambda \tau }  d \tau \psi(\theta) d \theta  d \tau \\
			% &= \int_{-\infty}^{0}e^{\lambda \tau } \int_0^\pi Lh(z)    \psi(\Theta(0,\tau,z)) G(\tau,z) d z  d \tau \\
			% &\leq \int_{-\infty}^{0} e^{(\lambda -\frac{1}{15})\tau} \int_0^\pi| Lh(z)|    |\psi(\Theta(0,\tau,z))|  d z  d \tau \\
			% &\leq \|\psi\|_{p'}  \|Lh \|_p \int_{-\infty}^{0} e^{(\lambda -\frac{1}{15}+\frac{p'}{15})\tau}   d \tau 
			% \end{align*}
		% where we used 
		% $$
		% \int_0^\pi |\psi(\Theta(0,\tau,z))|  ^{p'} dz = \int_0^\pi |\psi(\theta)|^{p'}G(-\tau,\theta) d \theta \leq e^{+\tau/15} \int_0^\pi |\psi(\theta)|^{p'} d \theta
		% $$
		% This shows that the change of variable is well defined for any $\lambda>0$.
		% }
	\subsection{Numerical investigation of the eigenvalues of the linear operator $T^\lambda$}

\paragraph{Description of the numerical method}
 
	Following Proposition~\ref{prop_carac_fonction_propre_A},  we aim to investigate numerically the existence of $\lambda$ with positive real part for which the operator $T^\lambda$ admits a non trivial fixed point, or equivalently, an eigenvalue equal to $1$. Using a finite element method, we will thus approach the eigenvalues of the operator $T^\lambda$ for several values of $\lambda$. 
	This leads us to address the following problem: given $p>2$ and $\lambda$ such that $\Re (\lambda)>1/15p$, find $ h\in L^p(0,\pi)$  and $\mu \in \C$ such that
	$$
	\int_0^\pi T^\lambda [h]  \psi = \mu \int_0^\pi h  \psi 
	$$
	for any test function $\psi \in L^{p'}(0,\pi)$. We use the set of  Legendre polynomials $\{h_k\}_{k\geq 1}$ as an orthonormal basis (in the $L^2(0,\pi)$ sense) in order to approach both $h\in L^p(0,\pi)$ and  $\psi\in L^{p'}(0,\pi)$. We introduce then the set $V_K=\text{vect} \{ h_i\}_{1\leq i \leq K}$.
	\medskip
	
	In order to handle the infinite integral defining $T^\lambda$, we first introduce the truncated operator $T^{\lambda,S}$ defined as 
	$$
	T^{\lambda,S}[h]= \int_{-S}^{0}  Lh (\Theta(\tau,0,\theta))e^{\lambda \tau }d \tau\, ,
	$$
	for some $S>0$. This new operator is an approximation of $T^\lambda$. Indeed, following the proof of  Proposition \ref{prop_T_lamda}, we have for any $h\in L^p$ and any $q\geq 1$ such that $ \frac{p}{q} \lambda - \frac{1}{15}>0$,
	\begin{align}\label{erreur_T^lambda_trunc}
		\|T^\lambda[h]- T^{\lambda,S}[h] \|_p &\leq  C_{q,p,\lambda}\left\|Lh \right\|_p    \left(\frac{e^{-S (\frac{p}{q} \lambda - \frac{1}{15})}}{\frac{p}{q} \lambda - \frac{1}{15}}\right)^{1/p} \notag \\
		&\leq  C_{q,p,\lambda} \left(\pi^{1/p} \left\|J \right\|+\frac{2}{15}  \right) \|h\|_p  \left(\frac{e^{-S (\frac{p}{q} \lambda - \frac{1}{15})}}{\frac{p}{q} \lambda - \frac{1}{15}}\right)^{1/p} 
	\end{align}
	where $C_{q,p,\lambda}=\big(\frac{2q'}{\lambda p'} \big)^{1/p'}$ and $\|J\|$ is the norm operator of $J$ from $L^p$ to $L^\infty$ given by \eqref{eq:borne_norm_J_infty}. Hence, for a fixed error $\eta>0$, we may choose $S \geq S(\lambda)>0$ with 
	\begin{equation}\label{eq:fomule_S_lamda}
		S(\lambda) =  \frac{-1 }{\frac{p}{q}\lambda -\frac{1}{15}}\ln \left[ \left(\frac{p}{q} \lambda -\frac{1}{15}\right) \left(\frac{\eta }{{C}_{q,p,\lambda} (\pi^{1/p} \|J\| + \frac{2}{15})}\right)^p\right]
	\end{equation}
	such that for any $ h\in L^p$ with $\|h\|_p=1$,
	$$\|T^\lambda[h]- T^{\lambda,S}[h] \|_p \leq \eta\, . $$
	\medskip
	
	We aim then to solve numerically the following problem: find $x\in \C^K$ and $\mu^{\lambda,S} \in \C$ such that 
	$$
	A^{\lambda,S} x  = \mu^{\lambda,S} x 
	$$
	where $A\in M_K(\R)$ is given by $$A^{\lambda,S}_{i,j}=\int_{(0,\pi)}T^{\lambda,S} [h_j] \, h_i.$$ Hence the problem reduces to finding the eigenvalues $\mu^{\lambda,S}$ and the corresponding eigenvectors of the matrix $A^{\lambda,S}$; more precisely, we are interested in the existence of the eigenvalue $\mu^{\lambda,S}=1$. 
	
	In order to compute the integral appearing in the definition of $A^{\lambda,S}_{i,j}$, we use the change of variable $z=\Theta(\tau,0,\theta)$ for a fixed $\tau \in(-S,0)$. This yields
	\begin{align*}
		&\int_0^\pi T^{\lambda,S} [h_j](\theta) \, h_i(\theta) d \theta  \\
		&= \int_0^\pi \left(\int_{-S}^0 L[h_j](\Theta(\tau,0,\theta))e^{\lambda \tau} d \tau \right)h_i(\theta) d \theta \\
		&=\int_{-S}^0  e^{\lambda \tau} \left(\int_0^\pi L[h_j](\Theta(\tau,0,\theta))e^{\tau}  h_i(\theta) d \theta \right) d \tau \\
		&=\int_{-S}^0  e^{\lambda \tau} \left(\int_0^\pi L[h_j](z)  h_i(\Theta(0,\tau,z))G(\tau,z) d z \right) d \tau \\
		&=\int_0^\pi L[h_j](z)\left( \int_{-S}^0  e^{\lambda \tau}   h_i(\Theta(0,\tau,z))G(\tau,z) d \tau \right) d z
	\end{align*}
	with $G$ defined in \eqref{eq:def_G}.
	We denote by $(\theta_p)_{0 \leq p \leq n_\theta}$ a subdivision of $[0,\pi]$ with $\theta_p=p\, \Delta_\theta$ and $\Delta_\theta= \frac{\pi}{n_\theta}$.
	In Python, we use the scipy package and more precisely
	\begin{itemize}
		\item the function integrate.dblquad to compute the integral defining $L[h_j](\theta_p)$ for $p=0 \cdots n_\theta$, $j=1, \ldots,  K$;
		\item the function  integrate.quad to compute the integral with respect to the variable $\tau \in (-S,0)$;
		\item the Simpson quadrature method to compute the integral with respect to the variable $z \in (0,\pi)$.
	\end{itemize}

\paragraph{Choice of parameters and numerical results}
 
{In all the following computations, we set $p=3$ and consider values of $\lambda$ greater or equal to $0.04$. In particular, $\lambda>\frac{1}{15 p }$ so we can apply Propositions~\ref{prop_T_lamda} and~\ref{prop_carac_fonction_propre_A}. We define $q=\frac{9}{10}15p \lambda>1$, which ensures that the condition $ \frac{p}{q} \lambda - \frac{1}{15}>0$ (required to derive the upper bound~\eqref{erreur_T^lambda_trunc}) is fulfilled.
To any tested value of $\lambda$, we will associate the error $E(\lambda)$ defined as
$$E(\lambda)=\min \left\{  |\mu^{\lambda,S}-1|, \, \mu^{\lambda,S}\in \C \text{ eigenvalue of } A^{\lambda,S}\right\}\, ,$$
where $S$ will be set uniformly with respect to the considered range of $\lambda$.

We start by considering $\lambda\in  (\lambda_j)_{0\leq j \leq n_\lambda}$ ranging from $0.04$ to $0.2$.
	Using \eqref{eq:fomule_S_lamda} we choose $S=5200>S(0.04)$ associated with the truncation error $\eta=10^{-3}$. A first numerical investigation shows that the error $E$
 reaches a minimal value in the sub-interval $[0.04,0.07]$ for $n_\theta=100,\, 200,\, 400 $ with $K=5,\ldots, 14$, see  Figure \ref{fig:plot_lamda1} for the error for $n_\theta=200$ (the results being similar for $n_\theta=100$ and $n_\theta=400$).
 
 We then perform a refined computation of $E$ for $\lambda\in [0.04,0.07]$, with a step $\Delta_\lambda=0.001$ and $K=6, 8, \ldots, 14$ (see Figure~\ref{fig:plot_lamda}). For $K$ between $8$ and $14$, we observe that the function $\lambda\mapsto E(\lambda)$ is likely to possess a local minimum. In order to estimate it more precisely, we use the optimization function minimize\_scalar in Python to compute the minimal value of $E$ and its argument $\lambda^\star$. The results of this procedure are represented in Figure~\ref{fig:plot_lamdastar}.

 Let us stress that for $K=7, 8, \ldots, 14$, the optimal value $E(\lambda^*)$ is of order $10^{-5}$, independently on the choice of parameter $n_\theta=100, 200, 400$. In this range of values of $K$, the function $K\mapsto \lambda^*$ appears regular and increasing, and tends to stabilize to a value located between $0.06$ and $0.065$. Note that up to $K=16$, the graph of the function $K\mapsto \lambda^*$ is invariant when one replaces $n_\theta=200$ by $n_\theta=400$. 
 
 However, when $K$ exceeds $16$, some numerical instabilities arise: $\lambda^*$ has no longer a smooth behaviour with respect to $K$, and moreover the associated error $E(\lambda^*)$ deviates from zero. This suggests that the accuracy of the computation, for fixed $n_\theta$, cannot be guaranteed if $K$ goes beyond a certain threshold. Hence, extending the presented results to larger values of $K$ (\emph{i.e.}, using Legendre polynomials of larger degrees) would certainly require to enhance precision of the quadrature methods used to compute the matrix $A^{\lambda,S}$. Another possibility would be to find a more adapted finite element basis. For instance, the set $\mathrm{Vect}(\cos^k)_{k\in \N}$ is naturally adapted to the problem since it is invariant under the operator $A$. Unfortunately, such a basis is not orthonormal and yields multiplication of the finite element matrix $A^{\lambda,S}$ by a dense matrix having small coefficients leading to additional numerical fluctuations.

\paragraph{Conclusion} Numerical investigation of the existence of a non vanishing eigenvalue for the operator $A$ has been carried out through the search of a fixed point for the operator $T^\lambda$ (in the case $\Re(\lambda)>0$), using a finite element method involving bases of Legendre polynomials of maximal degree $K$. In the range $7\leq K \leq 14$ and for several choices of the discretization parameter $n_\theta$, our computations point out the existence of a positive $\lambda^*$ such that the stiffness matrix $A^{\lambda^*,S}$ exhibit at least one complex eigenvalue whose distance from $1$ is of order $10^{-5}$. This striking behaviour can be interpreted as a strong indication of the existence of $\lambda$ greater than $1/{15p}$ (in case $p=3$) such that the continuous operator $T^{\lambda}$ admits non zero fixed points over $L^p(0,\pi)$.

Although the identified value of $\lambda^*$ tends to stabilize as $K$ increases, we are not able yet to achieve convergence due to numerical instabilities that could be explained by the presence of several integrals in the definition of $T^{\lambda}$, and in particular the non-local nature of the operator $J$.
Further investigations in this direction would certainly require greater computational resources as well as an improvement of the accuracy of the implemented integration methods.

 }

	\begin{figure}
		\centering
				\includegraphics[width=100mm]{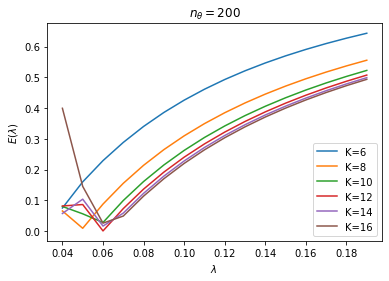}

		\caption{Evolution of the error $E(\lambda)$ in terms of $\lambda$ ranging from $0.04$ to $0.2$ with a step $\Delta_\lambda=0.01$ for $n_\theta=200$ and $K=6,8,\ldots,16$ .}
		\label{fig:plot_lamda1}
	\end{figure}

	\begin{figure}
		\centering
		\begin{tabular}{l }
			{
				\includegraphics[width=100mm]{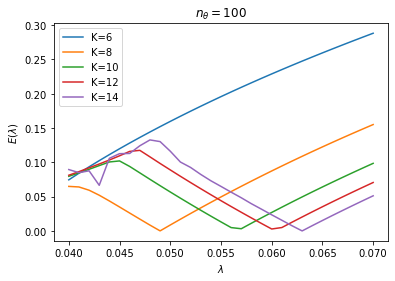}
			}\\
			{
				\includegraphics[width=100mm]{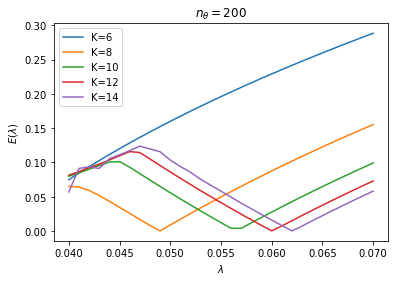}
			}\\
			{
				\includegraphics[width=100mm]{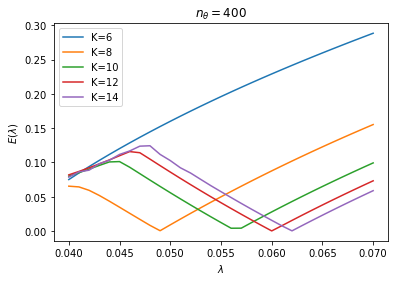}
			}
		\end{tabular}
		\caption{Evolution of the error $E(\lambda)$ in terms of $\lambda$ ranging from $0.04$ to $0.07$ with a step $\Delta_\lambda=0.001$.}
		\label{fig:plot_lamda}
	\end{figure}
	%%%%%%%%%%%%%%%%%%%%%%%%%%%%%%%%%%%%%%%%%%%%%%%%%
	\begin{figure}
		\centering
		\begin{tabular}{l }
			{
				\includegraphics[width=100mm]{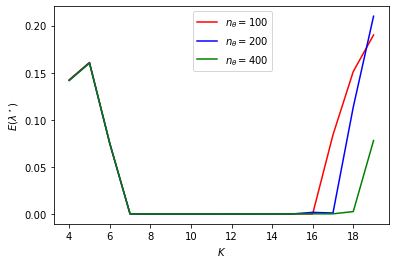}
			}\\
			{
				\includegraphics[width=100mm]{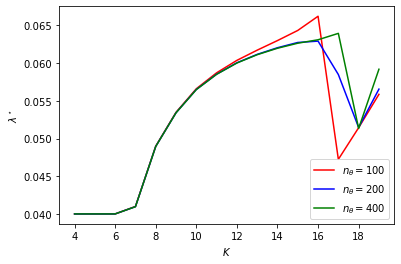}
			}
		\end{tabular}
		\caption{Values of $\lambda^\star$ and $E(\lambda^\star)$ for $\lambda\in [0.04,0.07]$ in terms of $K=4,\ldots,19$.}
		\label{fig:plot_lamdastar}
	\end{figure}

	\section{Simulations of transport-Stokes equation using a finite element method}\label{Sec:FEM}

	In order to perform accurate numerical simulations of the transport Stokes equation~\eqref{StokesR3}--\eqref{limInfini}, involving Stokes equation on $\R^3$, we rely on the observation that, if the initial density $\rho_0$ is axisymmetric, \emph{i.e.}\! invariant under rotation around the vertical axis $e_3$, then by rotational invariance of Stokes equation under such rotation, the density $\rho(t,\cdot)$ remains axisymmetric at any time $t>0$. This allows us to perform a reduction of dimension and to simulate axisymmetric, $2d$ versions of the different models under consideration in this section.

	\subsection{Fall of a suspension of particles in a cylindrical container}\label{Section:freefall}

	In this section, we aim at reproducing numerically the experiment conducted in~\cite{MNG}, where spherical suspensions of particles are injected in a vessel filled with a viscous fluid, and the evolution of the shape of the cloud is tracked. The authors observe that, if the number of injected particles is large enough, this shape is unstable: during the fall, particles leak from the rear of the cloud and form a vertical tail, which causes the emergence of a torus shape, before the cloud breaks into several droplets.

	Denoting by $\Omega\subset \R^3$ the domain occupied by the container, we consider the following transport-Stokes equation with homogeneous Dirichlet boundary condition on $\partial \Omega$:
	\begin{eqnarray}
		- \Delta u +\nabla p =- \rho\,  e_3 \quad \textrm{in }\Omega, \label{Stokes_Omega} \\
		\div u = 0  \quad \textrm{in }\Omega, \label{Incomp_Omega}\\
		\partial_t \rho + \div(\rho u) =0 \quad \textrm{in }\Omega, \label{Transport_Omega}\\
		\rho(t=0)= \rho_0 \quad \textrm{in }\Omega, \label{Init_rho_Omega}\\
		u = 0 \quad \textrm{on } \partial \Omega. \label{lim_dOmega}
	\end{eqnarray}

	\subsubsection{Axisymmetric formulation of the model}

	We use an axisymmetric setting and define $\Omega$ as the cylinder $\Omega=D(0,\widehat R)\times (0, \widehat L)$ where $\widehat L$ is a positive number and $D(0,\widehat R)$ is the disk of radius $\widehat R$ centred at the origin in $\R^2$. 
	%We assume that the initial density $\rho_0$ is independent on the angular variable $\theta$, and set $\rho_0=\rho_0(r,z)$.
	We also take $\rho_0$ as the indicator function of a ball $B(c_0,R)$:
	\[
	%\rho_0 = \frac{1_{B(c_0,R)}}{|B(c_0,R)|}\, ,
	\rho_0 = 1_{B(c_0,R)}\, ,
	\]
	where $c_0=(0,0,z_0)$ is a given point on the $e_3$ axis, with $0<z_0<\widehat L-R$, and $R$ is small with respect to $\widehat R$. We additionally set $\widehat \Omega = (0,\widehat R)\times (0,\widehat L)$.

	We introduce cylindrical coordinates $(r,\theta,z)$ defined by
	\[
	x_1=r\cos\theta,\quad x_2=r\sin\theta,\quad x_3=z.
	\]
	Since the domain $\Omega$ is axisymmetric, and the initial density $\rho_0$ is independent on the angular variable $\theta$, the external force $-\rho_0\, e_3$ in the momentum equation~\eqref{Stokes_Omega}  at time $t=0$ is independent on $\theta$. Hence, by classical results on Stokes equations in an axisymmetric domain (see for instance~\cite{BBD2006}), the vector field $u(t=0)$ is axisymmetric and has zero angular component, and the pressure $p(t=0)$ is independent on $\theta$. Using an explicit in time discretization of the transport equation~\eqref{Transport_Omega}, a recursive argument and letting the timestep go to zero, one can prove that $\rho$ remains independent on $\theta$ at any time. This yields, in turn, that $u$ is axisymmetric and has zero angular component, and that $p$ is independent on $\theta$ at any time.
	
	As a result, the velocity field of the fluid can be expressed at any time as $u = u_r\, e_r  + u_z\, e_z$ with $e_r = (\cos \theta, \sin \theta, 0)$, $e_z = e_3$, $\frac{\partial u_r}{\partial \theta} = \frac{\partial u_r}{\partial \theta} = 0$  and the momentum equation~\eqref{Stokes_Omega} reads in cylindrical coordinates
	\begin{equation}\label{Stokes_axisym}
		\left\{
		\begin{array}{rcll}
			\frac{\partial p}{\partial r}- \mu \left[ \frac{1}{r} \frac{\partial}{\partial r} \left ( r \frac{\partial u_r}{\partial r} \right)-\frac{u_r}{r^2} + \frac{\partial^2 u_r}{\partial z^2}  \right] &=&0& \textrm{in }\widehat \Omega,\\
			\frac{\partial p}{\partial z}- \mu \left [\frac{1}{r} \frac{\partial}{\partial r} \left ( r \frac{\partial u_z}{\partial r} \right) + \frac{\partial^2 u_z}{\partial z^2}  \right] &=&-\rho & \textrm{in }\widehat \Omega.\\
		\end{array}
		\right.
	\end{equation}
	Conditions~\eqref{Incomp_Omega} and~\eqref{Transport_Omega} respectively reduce to
	\begin{equation}\label{Incomp_axisym}
		\frac{1}{r} \frac{\partial}{\partial r} \left ( r u_r \right) + \frac{\partial u_z}{\partial z} = 0\quad \textrm{in }\widehat \Omega
	\end{equation}
	and
	\begin{equation}\label{Transport_axisym}
		\partial_t \rho + u_r \frac{\partial \rho}{\partial r}+  u_z \frac{\partial \rho}{\partial z} = 0 \quad \textrm{in }\widehat \Omega\,.
	\end{equation}

	%\subsubsection{Iterative algorithm}

	\subsubsection{Numerical method}\label{NumericalMethod:freefall}

	In order to approximate the density $\rho$ solution of the axisymmetric transport-Stokes system~\eqref{Stokes_axisym}--\eqref{Transport_axisym}, we rely on a Finite Element discretization of Stokes equation~\eqref{Stokes_axisym}-\eqref{Incomp_axisym} that we implement using the open source FreeFem++ software~\cite{FREEFEM}. The density $\rho$ is discretized in time and updated using an ALE method, with mesh velocity equal to the fluid velocity $u$~\cite{DecoeneMaury}.

	More specifically, let $[0,T]$ be the time interval for the simulation, $N\in \N^*$ and define the timestep $\delta_t = T/N$. Let $\widehat \Omega_0$ be an initial triangulation of the computation domain $\widehat \Omega$. With the ALE method, the mesh is moved at each iteration, hence we note $\widehat\Omega_n$ the mesh at iteration $n$. We set $u_n,p_n,\rho_n$ the approximate values of $u,p,\rho$ at instant $t_n=n\, \delta_t$, which are all FE functions defined on $\widehat\Omega_n$.

	The iterative procedure that we implement can be summarized as follows. Assuming that $\rho_n$ is given on $\widehat \Omega_n$, we define $u_n,p_n$ as the solution of~\eqref{Stokes_axisym}--\eqref{Incomp_axisym} with $\rho=\rho_n$. Then, the mesh $\widehat \Omega_{n+1}$ is obtained by moving each vertex of $\widehat\Omega_n$ by $\delta_t\, u_n$, and $\rho_{n+1}$ is defined on the new mesh $\widehat\Omega_n$ by pushing forward $\rho_n$. This step is performed with FreeFem using the command \texttt{movemesh}. At every instant $t_n$, we use $P_0$ elements for $\rho_n$, which enforces the density to remain a characteristic function throughout the simulation. As regards $u_n$ and $p_n$, we use $P_2$ and $P_1$ elements respectively, a Taylor-Hood approximation that ensures that the BBL inf-sup condition is fulfilled~\cite{Lee2011OnSA}.

	To create the initial mesh $\widehat\Omega_0$, we interpolate the indicator function $\rho_0$ on a very fine reference mesh $\widehat \Omega_{\textrm{ref}}$ (see Figure~\ref{Initial_mesh_tank}), using $P_1$ elements. We obtain a piecewise affine function $\rho_0^{\textrm{int}}$. Then, we build $\widehat\Omega_0$ by adapting $\widehat \Omega_{\textrm{ref}}$ to the metric defined by the Hessian matrix of $\rho_0^{\textrm{int}}$, using the command \texttt{adaptmesh}. The meshes $\widehat \Omega_{\textrm{ref}}$ and $\widehat\Omega_0$ are represented on Fig.~\ref{Initial_mesh_tank}.
	
	Since the mesh is moved at every iteration, in order to prevent the reversing of certain triangles, which would result in the creation of an invalid triangulation, we add a periodic remeshing step in the forementioned iterative procedure. This step is performed analogously to the construction of the initial mesh $\widehat\Omega_0$ described in the previous paragraph: $\rho_n$ is interpolated on the reference mesh, so as to define a piecewise affine function $\rho_n^{\textrm{int}}$ on $\widehat \Omega_{\textrm{ref}}$. The new computational mesh $\widehat \Omega_n$ is then obtained by adapting the reference mesh to the metric induced by the hessian matrix of $\rho_n^{\textrm{int}}$.

	\subsubsection{Results and comments} 
	
	The results of the simulation of the axisymmetric model of a fall of a spherical suspension in a container~\eqref{Stokes_axisym}--\eqref{Transport_axisym} are plotted in Figures~\ref{Fig:freefall} and~\ref{Fig:freefall_zoom}. The parameters are $R=40, L=10, T=200$ and $\delta_t=0.25$. One can observe that the initial spherical shape of the droplet is unstable: during the fall of the suspension, small clusters of particles dissociate from the rest of the droplet, in a region located around the axis of symmetry of the cylindrical domain $\Omega$. This phenomenon results in the appearing of tiny regions of density one, located above the main part of the droplet. As a consequence, the droplet loses matter around its axis of symmetry. 
	
	This behaviour appears consistent with the experimental observations provided in~\cite{MNG}, that reveal the appearance of a tail above the suspension, finally resulting in the formation of a torus shape.

	\begin{figure}
		\bc
		\includegraphics[width=2cm]
		{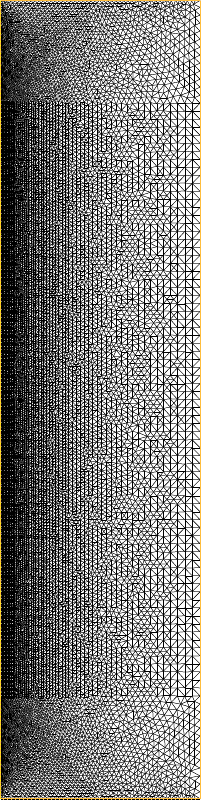} \qquad
		\includegraphics[width=2cm]
		{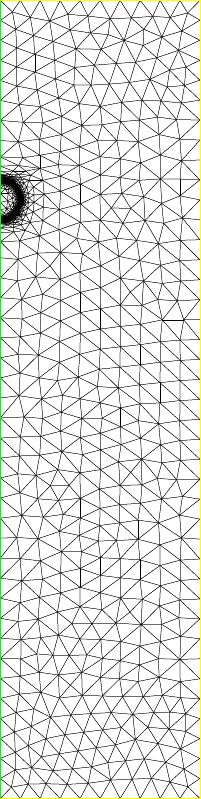}
		\ec
		\caption{Left: reference mesh $\widehat \Omega_{\textrm{ref}}$ of the domain $\widehat \Omega$ used for simulating the fall of a sphere in a container, modelled by the axisymmetric description~\eqref{Stokes_axisym}--\eqref{Transport_axisym}. Right: initial computational mesh $\widehat\Omega_0$.}\label{Initial_mesh_tank}
	\end{figure}

	\begin{figure}
		\bc
		\includegraphics[width=1.cm]
		{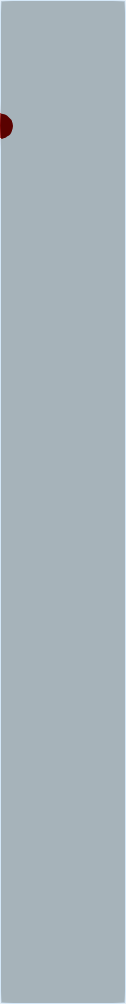} 
		\quad
		\includegraphics[width=1.cm]
		{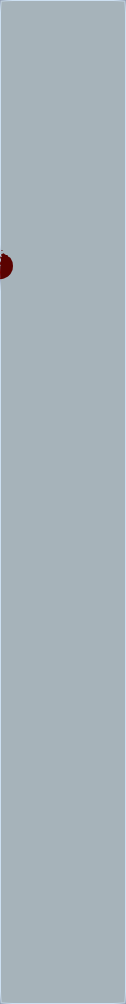}
		\quad
		\includegraphics[width=1.cm]
		{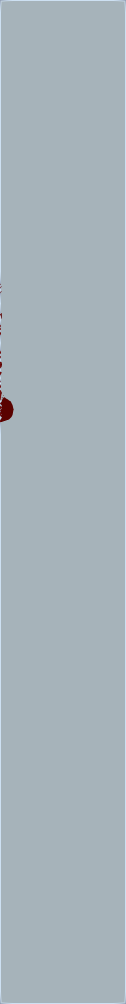}
		\quad
		\includegraphics[width=1.cm]
		{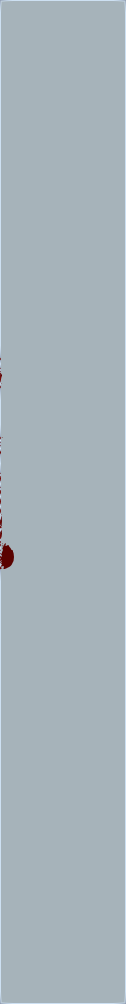} 
		\quad
		\includegraphics[width=1.cm]
		{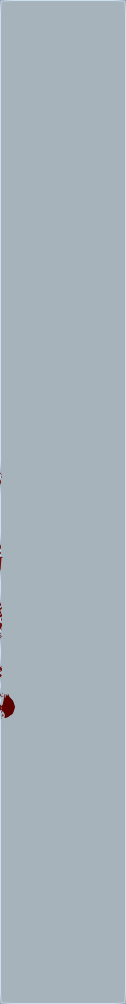} 
		\ec
		\caption{Density $\rho_n$ at times $t=0,50,100,150,200$ in the case of a fall in a cylindrical container. In lightgray, $\rho=0$; in red, $\rho=1$. Screenshots extracted from the video available \href{https://webusers.imj-prg.fr/~amina.mecherbet/tank_fall_video.mp4}{here} .}
		\label{Fig:freefall}
	\end{figure}

	\begin{figure}
		\bc
		% \includegraphics[width=4cm]
		% {pics/image_reservoir_zoom_T=50_rognee.png} 
		% \quad
		\includegraphics[width=5cm]
		{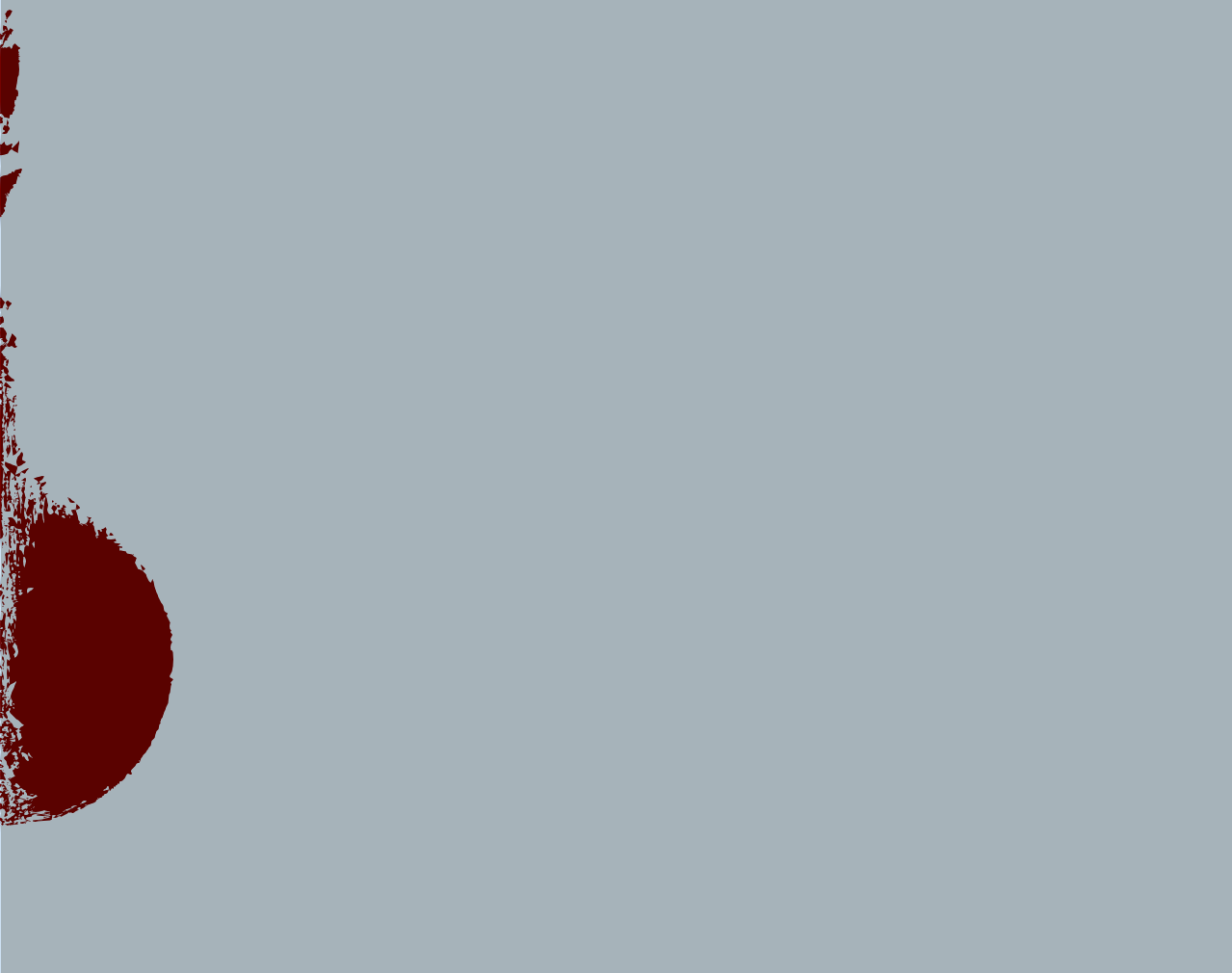} 
		\ec
		\caption{Density $\rho_n$ at time $t_n=100$, plotted after zooming  in the vicinity of the droplet, in the case of a fall in a cylindrical container. In lightgray, $\rho=0$; in red, $\rho=1$. 
			%Screenshot extracted from the video available here, displaying the result of the simulation on time interval $t\in [0,200]$.
		}
		\label{Fig:freefall_zoom}
	\end{figure}

	\subsection{Transport-Stokes equation in the whole space $\R^3$}
	
	In this section, we propose a different approach to test the stability of a spherical suspension of particles, described by a density $\rho$ solution to the transport-Stokes equation~\eqref{StokesR3}--\eqref{limInfini} in $\R^3$. To this aim, we combine a finite element method, similar to the one presented in Section~\ref{Section:freefall}, with the truncated domain-artificial boundary condition method.

	\subsubsection{Truncated domain and artificial boundary condition}

	We now start from a spherical suspension of particles located at the origin, and set 
	\[
	%\rho_0 = \frac{1_{B(c_0,R)}}{|B(c_0,R)|}\, ,
	\rho_0 = 1_{B(0,1)}\, .
	\]
	Let again be $[0,T]$ the time interval for the simulation, $N\in \N^*$ and $\delta_t=T/N$. We set $\rho_n$ the approximation of $\rho(t_n,\cdot)$ with $t_n=n\, \delta_t$, and define $u_n,p_n$ as the solution to the following Stokes problem over $\R^3$:
	\begin{eqnarray}
		-\Delta  u_n+ \nabla  p_n =-  \rho_n\,  e_3 \quad \textrm{in }\R^3, \label{Stokes_space_n} \\
		\div u_n = 0  \quad \textrm{in }\R^3, \label{Incomp_space_n}\\
		\lim_{|x|\ra \infty} u_n(x) = 0. \label{limInfiniSpacen}
	\end{eqnarray}
	Let $M_n$ be the center of mass of $\rho_n$, which is defined by
	\begin{equation}\label{Def:center_of_mass_Mn}
		M_n =  \frac{\int_{\R^3} \rho_n(x)\, x\, dx }{\int_{\R^3} \rho_n(x)\,  dx}.
	\end{equation}
	As recalled in the Introduction, in the case where the initial density $\rho_0$ is a spherical patch, the solution of the Transport-Stokes system~\eqref{StokesR3}--\eqref{limInfini} is a travelling wave. As a result, the position of $M_n$ is expected to move between iteration $n$ and $n+1$, with $M_{n+1}-M_n\approx \delta_t\, c^\star$ (at least for small values of $n$, for which the shape of the droplet remains close to the sphere). 
	
	However, the very specific behaviour of a spherical patch moved by the Transport-Stokes equation is a consequence of the symmetries of the problem. In order to preserve this symmetry during the iterative process, we introduce a truncated domain 
	\[
	\Omega^*=B(0, R^*),
	\]
	where $R^*\gg 1$ is a large radius with respect to the initial radius of the patch, and we translate the spatial coordinates so that $M_n$ coincides with the origin of $\R^3$. Then, we define $(u_n,p_n)$ as the solution to the following problem:
	\begin{eqnarray}
		-\Delta  u_n+ \nabla  p_n =-  \rho_n\,  e_3 \quad \textrm{in }\Omega^*, \label{Stokes_trunc} \\
		\div u_n = 0  \quad \textrm{in }\Omega^*, \label{Incomp_trunc}\\
		\frac{\partial u_n}{\partial \nu} -  p_n \nu +\frac{1}{R^*} u_n = 0\quad \textrm{on }\partial \Omega^*. \label{ArtificialBC}
	\end{eqnarray}
	Equation~\eqref{ArtificialBC} constitutes the artificial boundary condition, whose introduction allows one to derive error estimates of order $1/(R^*)^{3/2}$ on the solution $(u_n,p_n)$ to Stokes equation, in the case where the source term is compactly supported in the truncation domain $\Omega^*$. For more precise statements and details on this method, we refer to~\cite[Section 4]{GuirguisGunzburger87}.

	\subsubsection{Numerical method}

	To solve system~\eqref{Stokes_trunc}--\eqref{ArtificialBC}, we use a similar approach to the one described in Section~\eqref{NumericalMethod:freefall}, based on an axisymmetric formulation and a finite element method. We have represented in Figure~\ref{Fig:nonborne_refmesh} the reference mesh and the initial computational mesh, adapted to $\rho_0$. The major differences with the method from Section~\eqref{NumericalMethod:freefall} are the following.
	\begin{itemize}
		\item Since the vector field $u_n$ does not vanish on the boundary of the domain, we introduce a nonnegative truncature function $\zeta$ that vanishes near this boundary and is equal to $1$ in the semi-disk of radius $R^*-2$. At iteration $n$, the mesh vertices are then displaced from $\delta_t\, \zeta u_n$ instead of $\delta_t\, u_n$.
		\item After this first step of mesh updating, the density $\rho_n$ is pushed forward, so as to define a new density $\rho_{n+1}$ on the current mesh. Since its center of mass is not located at the origin, we compute numerically the position of $M_{n+1}$ by formula~\eqref{Def:center_of_mass_Mn} applied to $\rho_{n+1}$, and move the mesh once again by the displacement field $\zeta\, v_n$, where $v_n=-OM_{n+1}$ (interpreted as a vector field in $\R^2$). Finally, $\rho_{n+1}$ is pushed forward on the new mesh.
	\end{itemize}

	\begin{figure}
		\bc
		\includegraphics[width=6cm]
		{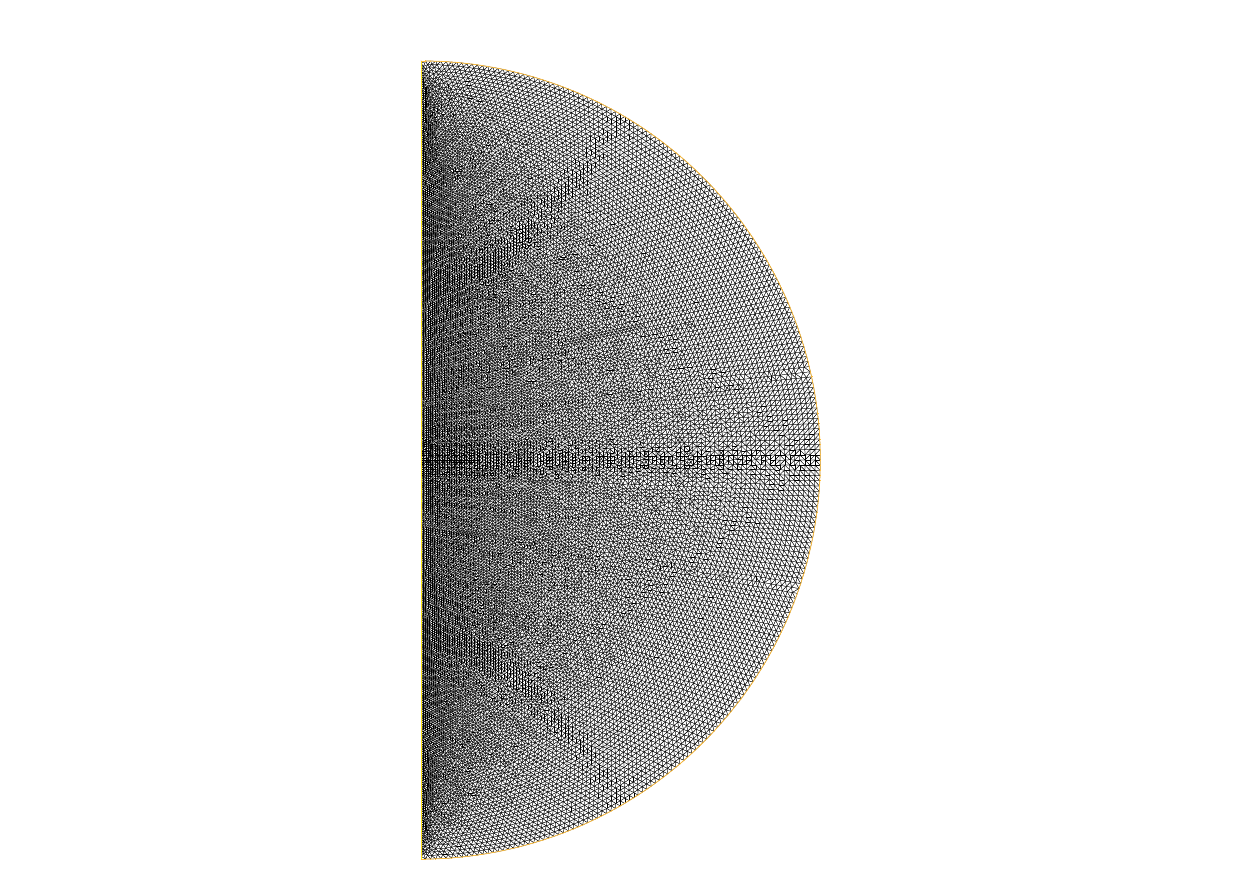} \quad
		\includegraphics[width=6cm]
		{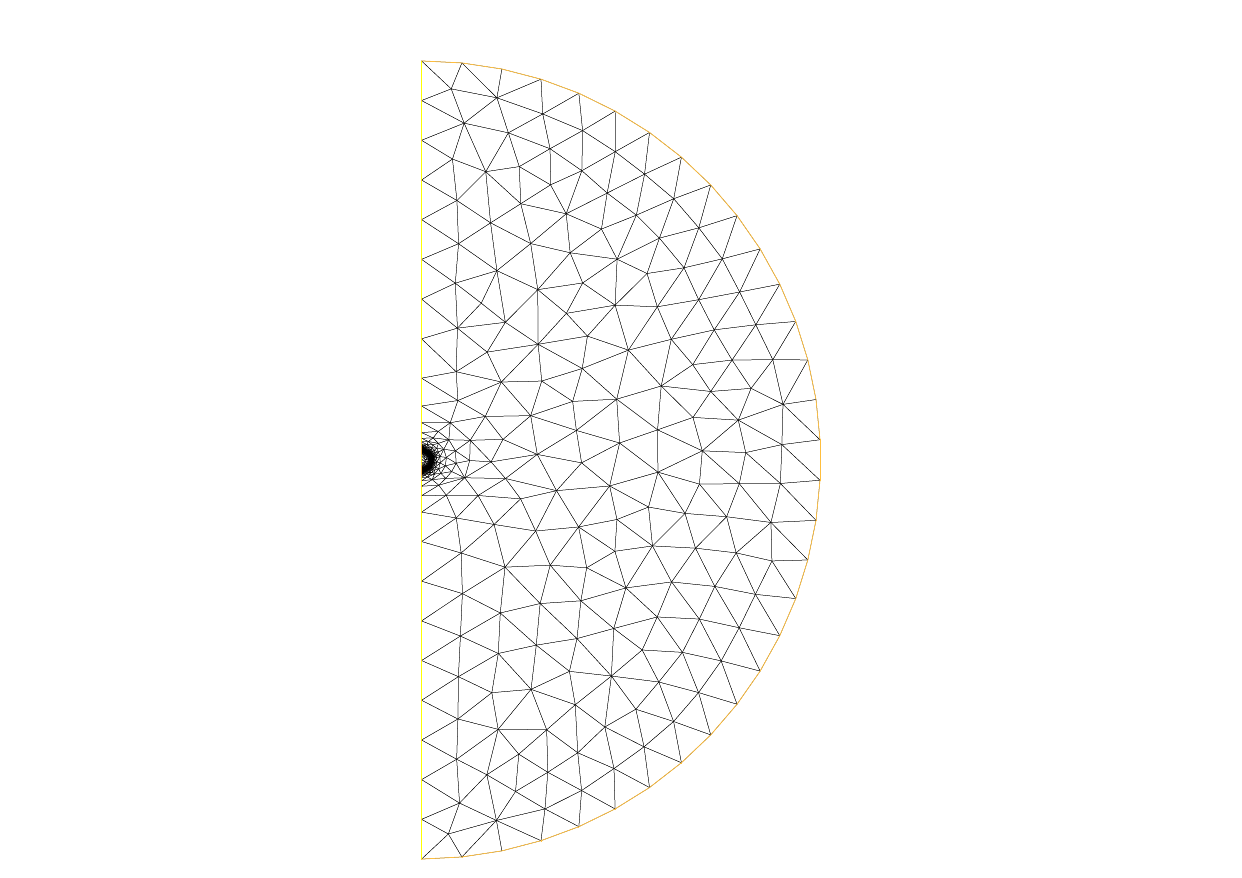}
		\ec
		\caption{Left: reference mesh used for simulating the Transport-Stokes equation in $\R^3$, in an axisymmetric setting, using artificial boundary conditions on the boundary of the truncated domain. Right: initial computational mesh, adapted to $\rho_0$.}\label{Initial_mesh_unbounded}
		\label{Fig:nonborne_refmesh}
	\end{figure}

	\subsubsection{Results and comments}

	We have plotted in Figure~\ref{Fig:res_nonborne} an example of simulation of the Transport-Stokes equation in $\R^3$, using an axisymmetric model and the numerical method presented above, with $R^*=40$ and $\delta_t = 0.025$. We observe an analogous behaviour as in the case of the fall of a droplet in a cylindrical tank: starting from a spherical patch, small clusters of density $1$ are released above the droplet, resulting in a loss of matter near its axis of symmetry, eventually conducting to the formation of a shape reminiscent of a torus. We emphasize that one can observe, at the level of the video, the departure from the closed Hadamard–Rybczynski toroidal circulation which are reminiscent of the observations described in \cite{MNG}.

	\begin{figure}
	\bc
	\includegraphics[width=4.cm]
	{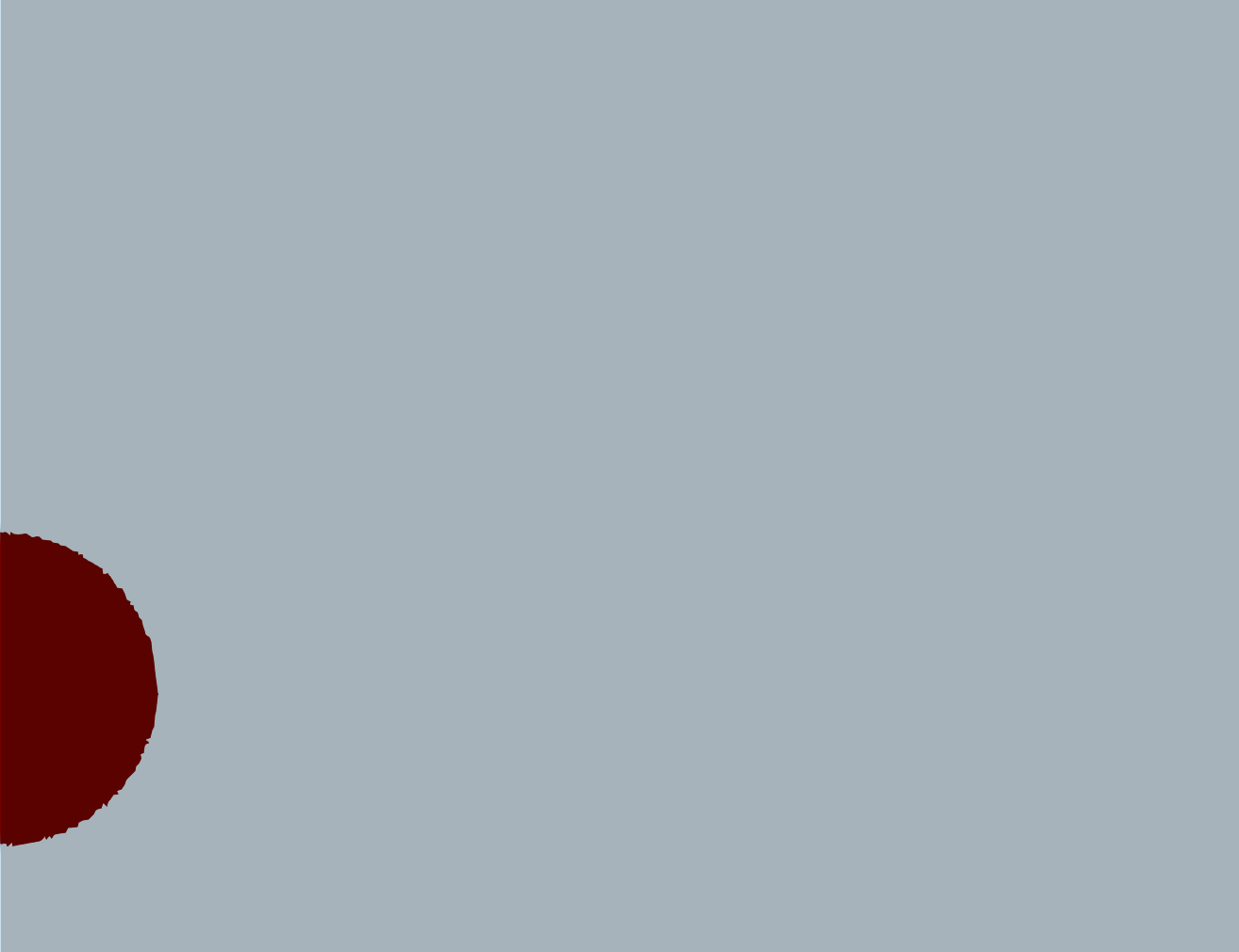} 
	\quad
	\includegraphics[width=4.cm]
	{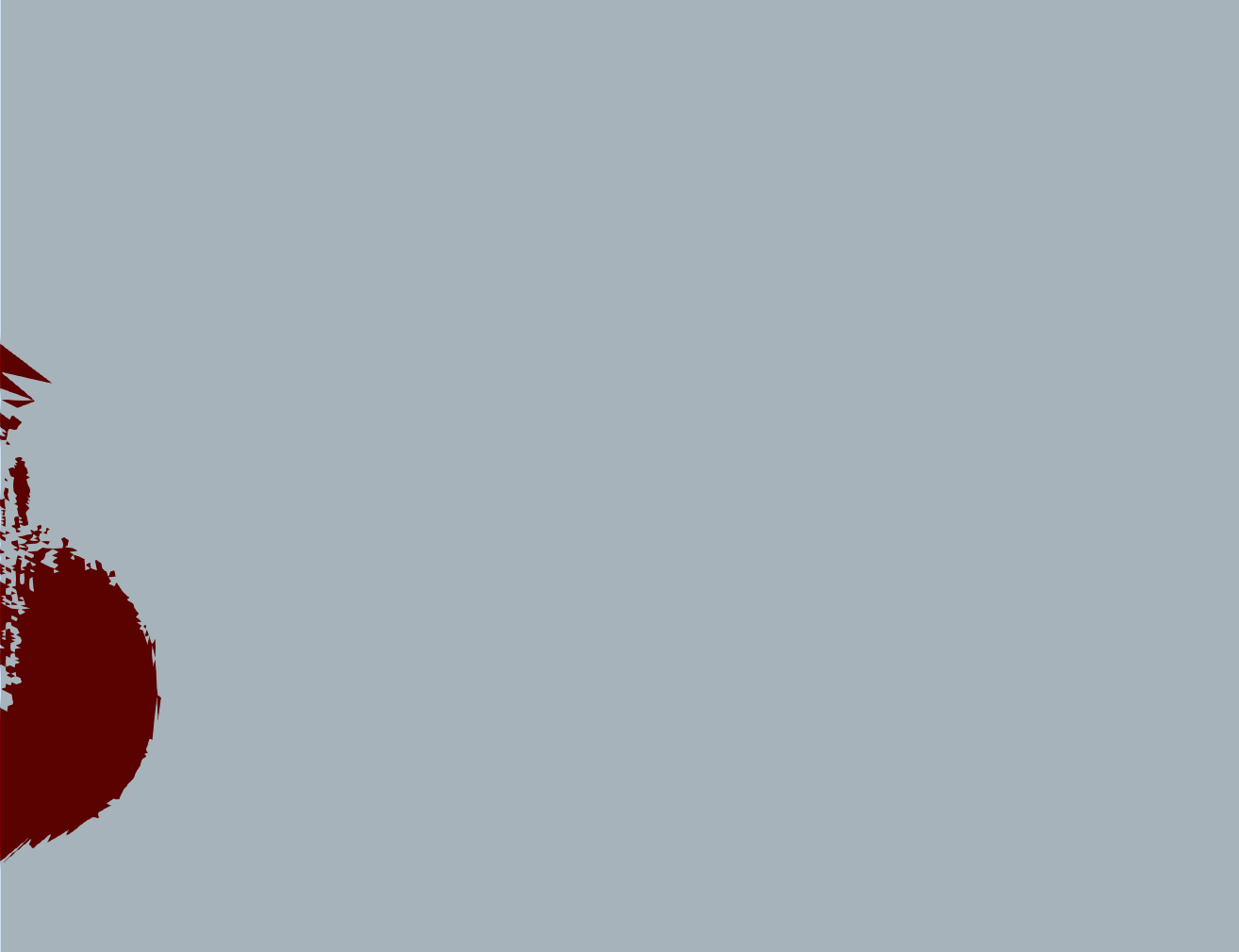}\\
	\includegraphics[width=4.cm]
	{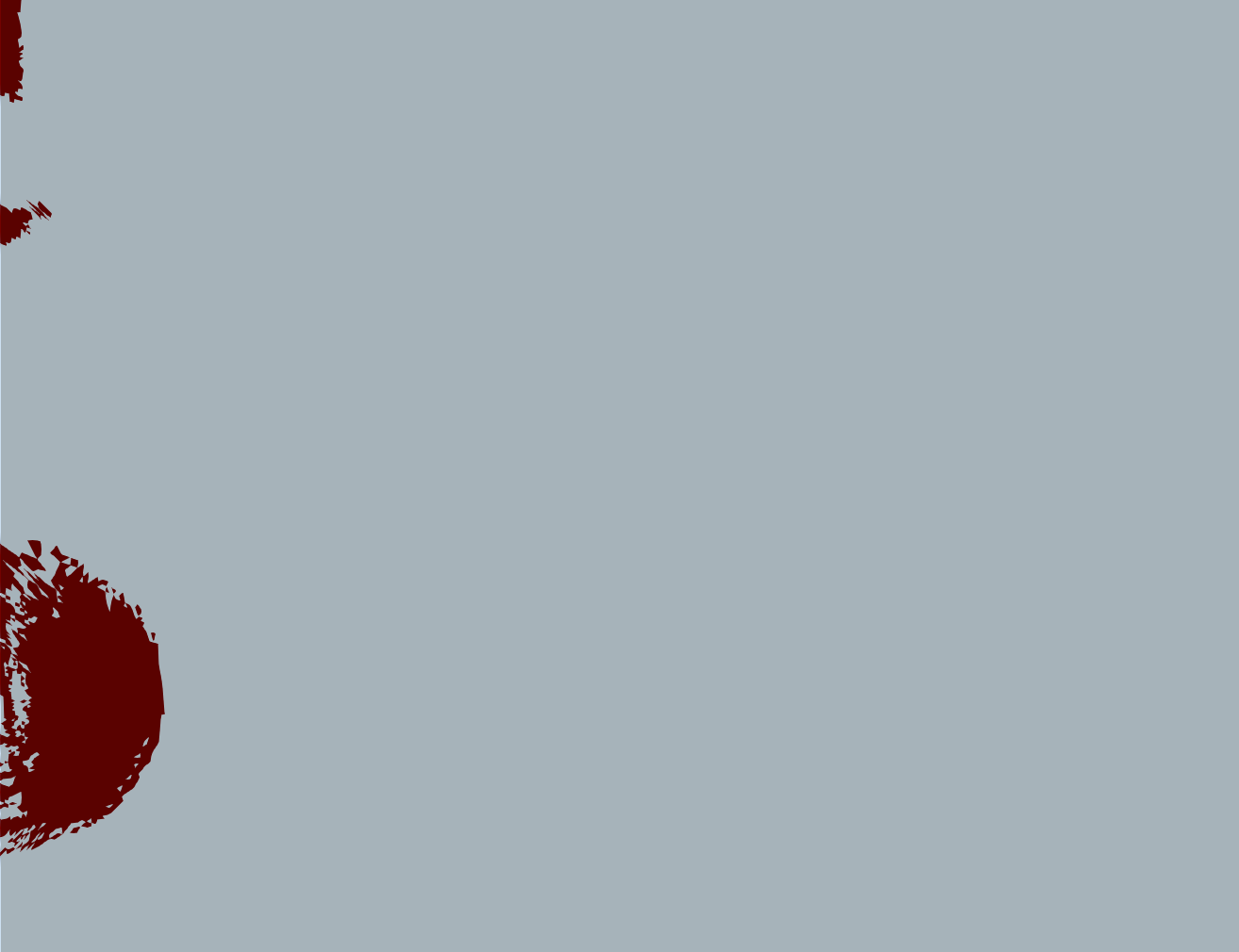}
	\quad
	\includegraphics[width=4.cm]
	{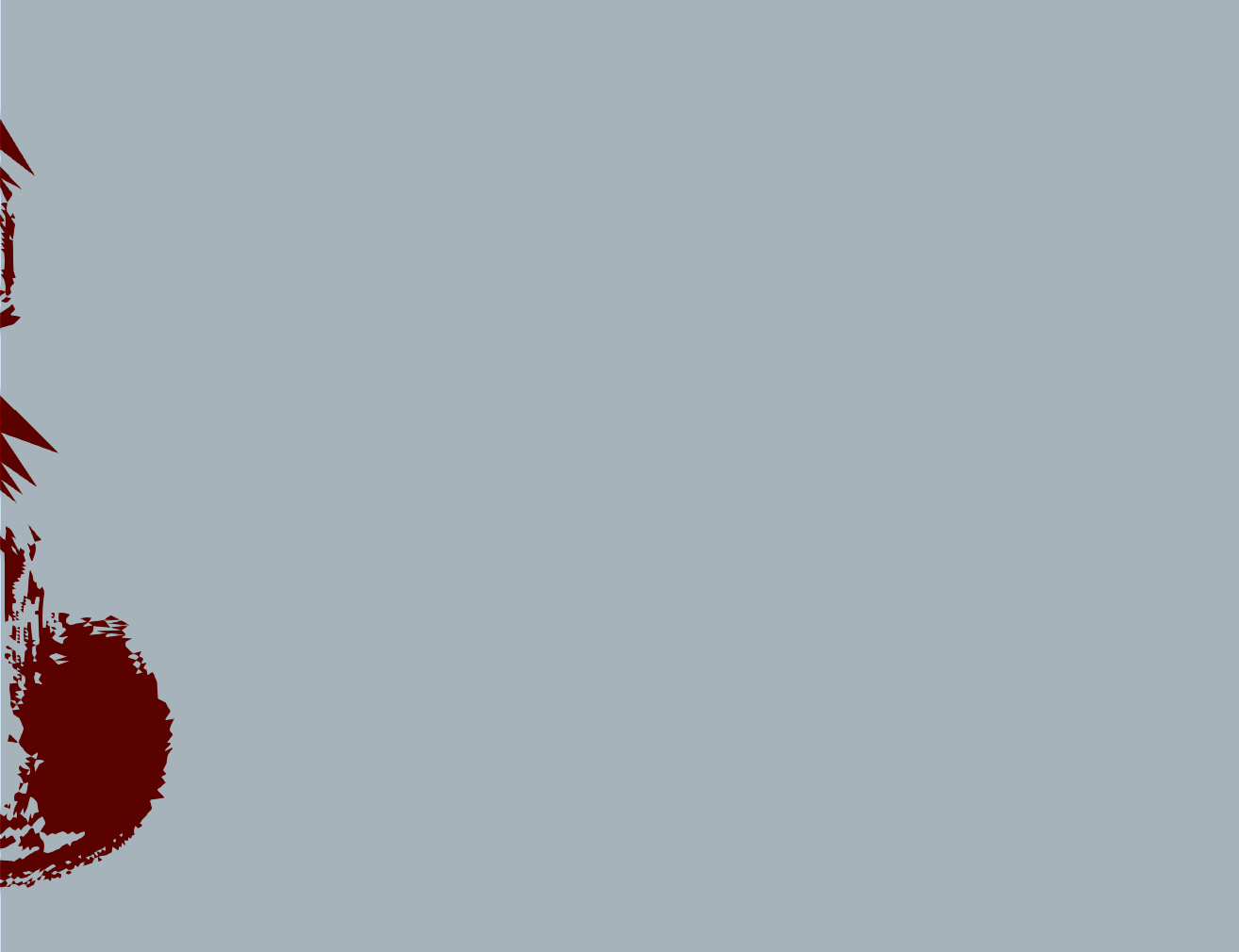} 
	\ec
	\caption{Density $\rho_n$ computed at times $t_n=0,50,100,150$ (from left to right and top to bottom), around the origin, using artificial boundary conditions on a truncated domain and an axisymmetric setting. In lightgray, $\rho=0$; in red, $\rho=1$. Screenshots extracted from the video available \href{https://webusers.imj-prg.fr/~amina.mecherbet/unbounded_domain_video_2}{here}.}
	\label{Fig:res_nonborne}
\end{figure}

	\appendix
	\section{Technical lemma}
	\begin{Lemma}\label{lemme_int_S2}
		
		For any $0 \leq \alpha<2$, there exists a positive constant $C>0$ satisfying
		$$
		\underset{ \theta \in[0,\pi]}{\sup} \bigg( \int_{[0,\pi]\times[0,2\pi]} \frac{ \sin (\bar \theta) }{\sqrt{2-2\sin(\theta) \sin(\bar \theta) \cos(\phi) - 2 \cos(\theta) \cos(\bar \theta) }^\alpha} d \bar \theta d  \phi \bigg) \leq C.
		$$
		
		%$$
		%\underset{ \theta \in[0,\pi]}{\sup} \left( \int_{[0,\pi]\times[0,2\pi]} \frac{ \sin (\bar \theta) }{|e(\bar \theta, \phi)-e( \theta,0) |^\alpha} d \bar \theta d  \phi+\int_{[0,\pi]\times[0,2\pi]} \frac{\sin(\bar \theta) d\bar  \theta d  \phi}{\sqrt{1-e( \theta,0)\cdot e(\bar \theta, \phi)^2}^\alpha} \right)  \leq C.
		%$$
	\end{Lemma}
	\begin{proof}
		Denoting by $e(\theta,\phi):=\begin{pmatrix}
			\sin (\theta) \cos(\phi) \\ \sin(\theta) \sin(\phi) \\ \cos(\theta)
		\end{pmatrix}$, the integral can be rewritten as
		$$
		\int_{[0,\pi]\times[0,2\pi]} \frac{ \sin (\bar \theta) }{|e(\bar \theta, \phi)-e( \theta,0) |^\alpha} d \bar \theta d  \phi
		$$
		and the proof is analogous to the one of \cite[Lemma B.1]{Mecherbet2020}.
		After using a change of variable in order to consider only the case $\theta=0$, we get 
		$$
		\int_{[0,\pi]\times[0,2\pi]} \frac{ \sin (\bar \theta) }{|e(\bar \theta, \phi)-e( 0,0) |^\alpha} d \bar \theta d  \phi=2\pi \int_0^\pi \frac{\sin(\bar \theta)}{\sqrt{2}^\alpha\sqrt{1-\cos(\bar \theta)}^\alpha} d \bar \theta
		$$
		which is integrable if $\alpha<2$.
	\end{proof}

	\section{Proof of Proposition~\ref{prop_carac_fonction_propre_A}}\label{Appendix:proofPropA}
	In order to lighten the proof let us assume again that $\lambda \in \R$, the proof works analogously for $\lambda \in \C$.
	Let $h\in L^p(0,\pi)$ be an eigenvector associated to
	$\lambda$, in the following weak sense: for any $\psi\in C^1_c(0,\pi)$,
	\begin{align}
		&-\frac{1}{15}\int_0^\pi \left(\cos (u)h(u)\psi(u)+ \sin (u)h(u)\psi'(u)\right)\, du + \int_0^\pi Lh(u)\psi(u)\, du\notag \\
		&= \lambda\int_0^\pi h(u)\psi(u)\, du\, .\label{WeakFormEigenvectorA}
	\end{align}
	The first step is to show that for a.e.\! $\theta\in (0,\pi)$ and $t<0$, 
	\begin{equation}\label{GoalProofEigenvalue}
		h(\Theta(t,0,\theta))e^{\lambda t} - h(\theta) = \int_0^t Lh(\Theta(\tau,0,\theta)) e^{\lambda \tau }\, d\tau\, .
	\end{equation}
	
	Let us take $\psi\in C^1_c(0,\pi)$ and fix $t<0$. We have using the change of variable $z=\Theta(\tau,0,\theta)$ with $\theta=\Theta(-\tau,0,z)$ and $d \theta= \partial_z \Theta(-\tau,0,z)\, dz$:
	\begin{align*}
		& \int_0^\pi \psi(\theta) \int_0^t Lh(\Theta(\tau,0,\theta)) e^{\lambda \tau }\, d\tau d \theta \\
		=& \int_0^\pi Lh(z)\int_0^t e^{\lambda \tau} \psi(\Theta(-\tau,0,z))\partial_z \Theta(-\tau,0,z)\,  d\tau dz\, .
	\end{align*}
	Setting $$f(z)=\int_0^t e^{\lambda \tau} \psi(\Theta(-\tau,0,z))\partial_z \Theta(-\tau,0,z) d\tau\, , $$ we see that $f$ is in $C^1_c(0,\pi)$ and satisfies
	\begin{align*}
		&\int_0^\pi Lh(z) f(z) dz \\
		= & \int_0^\pi h(z) \Big[ \frac{1}{15}\cos(z) f(z)+ \frac{1}{15} \sin(z) f'(z)\Big]\,  dz+\lambda \int_0^\pi h(z) f(z)\,  dz
	\end{align*}
	with 
	\begin{align*}
		\frac{1}{15} \sin(z) f'(z)  =  &  \frac{1}{15} \sin(z) \left(\int_0^t e^{\lambda \tau} \psi'(\Theta(-\tau,0,z))(\partial_z \Theta(-\tau,0,z))^2 d\tau \right)\\
		& +\frac{1}{15} \sin(z)\int_0^t e^{\lambda \tau} \psi(\Theta(-\tau,0,z))\partial^2_z \Theta(-\tau,0,z) d\tau\, .
	\end{align*}
	However, since $\frac{1}{15} \sin(z)\partial_z \Theta(-\tau,0,z)=\partial_\tau \Theta(-\tau,0,z) $, integrating by parts yields
	\begin{align*}
		\frac{1}{15} \sin(z) f'(z)= &  \int_0^t e^{\lambda \tau} \left[\partial_\tau \tau \mapsto \psi(\Theta(-\tau,0,z))\right]\partial_z \Theta(-\tau,0,z) d\tau \\
		& +\frac{1}{15} \sin(z)\int_0^t e^{\lambda \tau} \psi(\Theta(-\tau,0,z))\partial^2_z \Theta(-\tau,0,z) d\tau\\
		= &-\lambda  \int_0^t e^{\lambda \tau}  \psi(\Theta(-\tau,0,z))\partial_z \Theta(-\tau,0,z) d\tau\\
		&- \int_0^t e^{\lambda \tau}   \psi(\Theta(-\tau,0,z))\partial_\tau \partial_z \Theta(-\tau,0,z) d\tau\\
		&+e^{\lambda t}  \psi(\Theta(-t,0,z))\partial_z \Theta(-t,0,z)-\psi(z)\\
		&+\frac{1}{15} \sin(z)\int_0^t e^{\lambda \tau} \psi(\Theta(-\tau,0,z))\partial^2_z \Theta(-\tau,0,z) d\tau\\
		= &-\lambda f(z) - \int_0^t e^{\lambda \tau}   \psi(\Theta(-\tau,0,z))\partial_\tau \partial_z \Theta(-\tau,0,z) d\tau\\
		&+e^{\lambda t}  \psi(\Theta(-t,0,z))\partial_z \Theta(-t,0,z)-\psi(z)\\
		&+\frac{1}{15} \sin(z)\int_0^t e^{\lambda \tau} \psi(\Theta(-\tau,0,z))\partial^2_z \Theta(-\tau,0,z) d\tau\, .
	\end{align*}
	As a result,
	\begin{align*}
		&\frac{1}{15}\cos(z) f(z)+ \frac{1}{15} \sin(z) f'(z)\\
		& = -\lambda f(z) +e^{\lambda t}  \psi(\Theta(-t,0,z))\partial_z \Theta(-t,0,z)-\psi(z)\\
		&\quad +\int_0^t e^{\lambda \tau}   \psi(\Theta(-\tau,0,z))\\
		& \quad \times\left(\frac{1}{15}\big[\cos(z)\partial_z \Theta(-\tau,0,z)+\sin(z)\partial^2_z \Theta(-\tau,0,z)\big] - \partial_\tau \partial_z \Theta(-\tau,0,z)\right)  d\tau\, .
	\end{align*}
	Last line vanishes by deriving with respect to $z$ the relation 
	$$\frac{1}{15} \sin(z)\partial_z \Theta(-\tau,0,z)=\partial_\tau \Theta(-\tau,0,z) $$ so that the previous equality reduces to
	\begin{multline*}
	\frac{1}{15}\cos(z) f(z)+ \frac{1}{15} \sin(z) f'(z)= \\ -\lambda f(z) +e^{\lambda t}  \psi(\Theta(-t,0,z))\partial_z \Theta(-t,0,z)-\psi(z)\, .
	\end{multline*}
	
	Gathering the previous computations, we see that
	\begin{align*}
		& \int_0^\pi \psi(\theta) \int_0^t Lh(\Theta(\tau,0,\theta)) e^{\lambda \tau }\, d\tau d \theta \\
		& = \int_0^\pi Lh(z)f(z) d z\\
		&= \int_0^\pi  h(z)\left( -\lambda f(z) +e^{\lambda t}  \psi(\Theta(-t,0,z))\partial_z \Theta(-t,0,z)-\psi(z)\right) dz\\
		&\quad + \lambda \int h(z) f(z)dz\\
		&=\int_0^\pi h(\Theta(t,0,\theta))e^{\lambda t} \psi(\theta) d \theta - \int_0^\pi h(z)\psi(z) dz \, .
	\end{align*}
	This proves~\eqref{GoalProofEigenvalue}, which by definition
	~\eqref{Def:F}, can be rewritten
	\begin{equation}\label{AlmostDonePointFixe}
		h(\theta)-h(\Theta(t,0,\theta)) = F(t,\theta)\, .
	\end{equation}
	By the proof of item~\ref{item1_eq_vp}, $F(t,\cdot)$ converges to $T^{\lambda}[h]$ in $L^p(0,\pi)$. Also, using once again the change variable $z=\Theta(t,0,\theta)$, there holds the upper estimate
	\begin{align*}
		\int_0^\pi |h(\Theta(t,0,\theta))|^p e^{\lambda p t}\, d\theta & \leq \int_0^\pi |h(z)|^p e^{(\lambda p-\frac{1}{15})t}\, d\theta \\
		& \leq e^{(\lambda p-\frac{1}{15})t} \|h\|_{p}^p,
	\end{align*}
	which implies that $h(\Theta(t,0,\cdot))$ converges to $0$ in $L^p(0,\pi)$ as $t\rightarrow -\infty$, since $\lambda p-\frac{1}{15}>0$.
	Hence, we can pass to the limit in $L^p$ in relation~\eqref{AlmostDonePointFixe} to obtain~\eqref{FixedPointTlambda}
	.
	\medskip
	
	It remains to prove that any function $h\in L^p(0,\pi)$ satisfying~\eqref{FixedPointTlambda} is indeed an eigenvalue of operator $A$. Take $\psi\in C^1_c(0,\pi)$, we want to show that relation~\eqref{WeakFormEigenvectorA} holds true. This results from the following series of computations. Introducing the new variable $z=\Theta(\tau,0,\theta)$, which yields $\theta=\Theta(0,\tau,z)=\Theta(-\tau,0,z)$ with $d\theta= \partial_z \Theta(-\tau,0,z)\, dz$, we get
	\begin{align*}
		& -\int_0^\pi h(\theta) [ \frac{1}{15}\cos(\theta) \psi(\theta)+ \frac{1}{15} \sin(\theta) \psi'(\theta)] d \theta \\
		=&
		-\int_0^\pi \int_{-\infty}^0 Lh(\Theta(\tau,0,\theta))e^{\lambda \tau}  [ \frac{1}{15}\cos(\theta) \psi(\theta)+ \frac{1}{15} \sin(\theta) \psi'(\theta)] d \theta d \tau\\
		=& - \int_{-\infty}^0\int_0^\pi Lh(z)e^{\lambda \tau}  \Big[ \frac{1}{15}\cos(\Theta(-\tau,0,z)) \psi(\Theta(-\tau,0,z))\\
		& +\frac{1}{15} \sin(\Theta(-\tau,0,z)) \psi'(\Theta(-\tau,0,z))\Big] 
		\partial_z \Theta(-\tau,0,z)d z d \tau\, .
	\end{align*}
	Using the relation $\partial_\tau \Theta(-\tau,0,z) =\frac{1}{15} \sin(\Theta(-\tau,0,z))$ and integration by parts, last relation reduces to
	\begin{align*}
		& -\int_0^\pi h(\theta) [ \frac{1}{15}\cos(\theta) \psi(\theta)+ \frac{1}{15} \sin(\theta) \psi'(\theta)] d \theta\\
		=&- \int_{-\infty}^0\int_0^\pi Lh(z)e^{\lambda \tau}  \partial_z\left(  z \mapsto \frac{1}{15}\sin(\Theta(-\tau,0,z))\right) \psi(\Theta(-\tau,0,z))\\
		& - \int_{-\infty}^0\int_0^\pi Lh(z)e^{\lambda \tau}  \partial_\tau\left( \tau \mapsto \psi(\Theta(-\tau,0,z))\right) 
		\partial_z \Theta(-\tau,0,z)d z d \tau\\ 
		=&- \int_{-\infty}^0\int_0^\pi Lh(z)e^{\lambda \tau}  \partial_z \partial_\tau \Theta(-\tau,0,z)\psi(\Theta(-\tau,0,z))\\
		&+\lambda \int_{-\infty}^0\int_0^\pi Lh(z)e^{\lambda \tau}  \psi(\Theta(-\tau,0,z))
		\partial_z \Theta(-\tau,0,z)d z d \tau\\ 
		&+ \int_{-\infty}^0\int_0^\pi Lh(z)e^{\lambda \tau}  \psi(\Theta(-\tau,0,z))
		\partial_\tau \partial_z \Theta(-\tau,0,z)d z d \tau\\ 
		&-\int_0^\pi Lh(z) \psi(z) dz\\
		=&\lambda \int_{-\infty}^0\int_0^\pi Lh(z)e^{\lambda \tau}  \psi(\Theta(-\tau,0,z))
		\partial_z \Theta(-\tau,0,z)d z d \tau\\ 
		&-\int_0^\pi Lh(z) \psi(z) dz\\
		=&\lambda \int_0^\pi \left( \int_{-\infty}^0 Lh(\Theta(\tau,0,\theta)) e^{\lambda \tau}d \tau \right)\psi(\theta) d \theta - \int_0^\pi Lh(z) \psi(z) d z \\
		=&\lambda \int_0^\pi h(\theta) \psi(\theta) d \theta - \int_0^\pi Lh(z) \psi(z) d z \, .
	\end{align*}

	\bibliographystyle{plain}

	%\bibliography{bibTranStokes}
	\section*{Acknowledgements}
	Authors would like to thank the organizers of the Sophie Germain PDE seminar in Université Paris Cit\'e, where their collaboration was initiated.
	They are also grateful to David Gérard-Varet and Richard H\"ofer for fruitful discussions related to this topic. 
	%A.M. is supported by the SingFlows project, Grant ANR-18-CE40-0027 of the French National Research Agency (ANR).

\end{document}